\documentclass[a4paper, english]{amsart}

\usepackage{etex}

\usepackage{amsmath}
\usepackage{amssymb}
\usepackage{enumerate}
\usepackage{graphicx}
\usepackage{comment} 

\usepackage{lmodern}
\usepackage{booktabs}
\usepackage[dvipsnames,svgnames,x11names,hyperref]{xcolor}
\usepackage{url}
\usepackage[pagebackref,colorlinks,citecolor=Mahogany,linkcolor=Mahogany,urlcolor=Mahogany,filecolor=Mahogany,hypertexnames=false]{hyperref}

\usepackage{tikz}
\usepackage{tikz-cd}

\newtheorem{thm}{Theorem}[section]
\newtheorem{cor}[thm]{Corollary}
\newtheorem{lem}[thm]{Lemma}

\theoremstyle{definition}

\theoremstyle{remark}
\newtheorem{rem}[thm]{Remark}
\newtheorem{ex}[thm]{Example}
\numberwithin{equation}{section}

\newcommand{\bC}{\mathbb{C}}

\newcommand{\bP}{\mathbb{P}}
\newcommand{\bQ}{\mathbb{Q}}
\newcommand{\bR}{\mathbb{R}}

\newcommand{\bZ}{\mathbb{Z}}

\newcommand{\cH}{\mathcal{H}}

\newcommand\lra{\longrightarrow}

\newcommand\trf{\mathrm{trf}}
\newcommand\Diff{\mathrm{Diff}}

\newcommand\SL{\mathrm{SL}}
\newcommand\Sp{\mathrm{Sp}}

\renewcommand\P{\bP}

\newcommand\KMU{\kappa^{MU}}

\newcommand\MU{\mathrm{MU}}
\newcommand\K{\mathrm{K}}
\newcommand\E{\mathrm{E}}
\renewcommand\H{\mathrm{H}}
\newcommand\ch{\mathrm{ch}}

\title[Mumford--Morita--Miller classes]{Mumford--Morita--Miller classes\\in generalised cohomology theories}
\author{Oscar Randal-Williams}
\email{o.randal-williams@dpmms.cam.ac.uk}
\address{Centre for Mathematical Sciences\\
Wilberforce Road\\
Cambridge CB3 0WB\\
UK}

\begin{document}
\begin{abstract}
The Mumford--Morita--Miller classes on the moduli space of Riemann surfaces may be defined not only in ordinary cohomology but in any complex-oriented cohomology theory. We begin their investigation, emphasising the new phenomena that arise.
\end{abstract}
\maketitle

\setcounter{tocdepth}{1}
\tableofcontents

\section{Premise}

Given an oriented surface bundle $\pi : E \to B$, i.e.\ a smooth fibre bundle with fibres closed 2-dimensional manifolds equipped with compatible orientations, the Mumford--Morita--Miller \cite{Mumford, MoritaChar, Miller} classes 
$$\kappa_i(\pi) \in \H^{2i}(B;\bZ)$$
are defined by the formula $\kappa_i(\pi) := \pi_!(e(T_\pi E)^{i+1})$, the pushforward along $\pi$ of the $(i+1)$st power of the Euler class of the tangent bundle of $E$ along the fibres of $\pi$. We may also say that the orientation of the 2-dimensional real vector bundle $T_\pi E \to E$ equips it with a complex structure (unique up to isomorphism), for which $c_1(T_\pi E) = e(T_\pi E)$. Hence $\kappa_i(\pi)$ may equally well be defined as the pushforward along $\pi$ of the $(i+1)$st power of $c_1(T_\pi E)$.

In order to make this definition, all we have used is that $T_\pi E$ has a first Chern class in integral cohomology, and that the map $\pi$ has an associated pushforward map in integral cohomology. In fact, this data exists in many generalised cohomology theories, and exists universally in complex cobordism theory, denoted $\MU^*(-)$. There is a (Conner--Floyd--)Chern class $c_1^{MU}(T_\pi E) \in \MU^{2}(E)$, and a pushforward map $\pi_!^{MU} : \MU^*(E) \to \MU^{*-2}(B)$, using which we may define
$$\kappa_i^{MU}(\pi) := \pi_!^{MU}(c_1^{MU}(T_\pi E)^{i+1}) \in \MU^{2i}(B).$$

Under the natural cycle map $\MU^*(-) \to \H^*(-;\bZ)$ the classes $\kappa_i^{MU}(\pi)$ map to $\kappa_i(\pi)$. However, for any complex-oriented cohomology theory $E$ there is by definition a cycle map $\MU^*(-) \to \E^*(-)$, defining $\kappa_i^E(\pi) \in \E^{2i}(B)$. This yields $K$-theory Mumford--Morita--Miller classes $\kappa_i^K(\pi) \in \K^0(B)$, as well as classes in the various ``designer" cohomology theories studied in chromatic homotopy theory. 

This paper begins the investigation of these generalised Mumford--Morita--Miller classes $\kappa_i^{E}$. There are several new phenomena which arise when this broader perspective is taken, and our goal is to indicate several avenues which may deserve further investigation, rather than to present any particular result.

\vspace{1ex}
\noindent\textbf{Acknowledgments.} Most of the ideas contained here were developed in 2013-14, and some early thoughts in the direction of Section \ref{sec:MUQ} were presented at the conference ``Cohomology of the moduli space of curves'' at Zürich in 2013. More recently I have benefited from conversations with Irakli Patchkoria, who has studied Euler characteristics of mapping class groups in generalised cohomology theories \cite{Patchkoria1,Patchkoria2}. Claude Opus 5.5 was used to produce Sage \cite{sage} code verifying the calculations, and to proofread a draft.

\section{Recollections and conventions}

\subsection{The moduli stack of Riemann surfaces}\label{sec:MgBackground}

We will be concerned with the generalised cohomology of the homotopy type of the moduli stack $\mathcal{M}_g$ of Riemann surfaces of genus $g \geq 2$. For us this is most conveniently described as the quotient
$$\mathcal{M}_g := \mathcal{T}_g /\!\!/ \Gamma_g$$
of the action of the mapping class group $\Gamma_g := \pi_0 \mathrm{Diff}^+(\Sigma_g)$ of the genus $g$ surface on Teichm{\"u}ller space $\mathcal{T}_g$, formed in the category of topological stacks. It carries the universal family of Riemann surfaces
$$\pi : \mathcal{C}_g \lra \mathcal{M}_g.$$
For $g=1$ we instead consider the moduli stack $\mathcal{M}_{1}^1$ of genus 1 Riemann surfaces equipped with a choice of point, i.e. elliptic curves. This has a similar description, where $\mathcal{T}_{1}^1$ is now identified with the complex upper half-plane, upon which $\Gamma_{1}^1 = \SL(2,\bZ)$ acts by M{\"o}bius transformations. It carries the universal family of elliptic curves
$$\pi : \mathcal{C}_{1}^1 \lra \mathcal{M}_{1}^1.$$

A topological stack presents a homotopy type \cite{Noohi, EbertStack}: in the case at hand, $\mathcal{T}_g$ is contractible so the homotopy type of $\mathcal{M}_g$ is simply the classifying space $B\Gamma_g$ of the mapping class group; similarly the homotopy type of $\mathcal{M}_{1}^1$ is the classifying space $B\SL(2, \bZ)$. So we will equivalently be discussing the generalised group cohomology of these groups.

The action of $\Gamma_g$ on Teichm{\"u}ller space $\mathcal{T}_g$ is a universal proper action, i.e.\ the fixed set of a subgroup $G \leq \Gamma_g$ is contractible if $G$ is finite and empty otherwise (``$\mathcal{T}_g$ is an $\underline{E}\Gamma_g$''). This is discussed by Mislin \cite[Section 3]{Mislin}, who then explains how the work of Broughton \cite[Theorem 2.7]{Broughton} shows that $\mathcal{T}_g$ equivariantly deformation retracts to a finite $\Gamma_g$-CW-complex. In conclusion, there is a finite $\underline{E}\Gamma_g$. We also record that $\Gamma_g$ has a torsion-free subgroup of finite index \cite[Theorem 6.9]{FarbMargalit}.

Similarly, there is a finite $\underline{E}\SL(2, \bZ)$ (which is a tree) and $\SL(2,\bZ)$ has a torsion-free subgroup of finite index.

\subsection{Topologies and completeness}

Let $\E$ be a spectrum. If $X$ is a CW-complex with finite skeleta, the $\E$-cohomology of $X$ has a multiplicative filtration 
$$F^p\E^*(X) := \mathrm{Ker}(\E^*(X) \to \E^*(X^{(p-1)})).$$
 The cosets of these subgroups define a basis for a topology on $\E^*(X)$. The Milnor sequence takes the form
$$0 \lra \bigcap_{p \geq 0} F^p\E^*(X) \lra \E^*(X) \lra \lim_p \E^*(X^{(p)}) \lra 0,$$
so this topology is (complete and) Hausdorff if and only if $\bigcap_{p \geq 0} F^p\E^*(X) = \lim^1 \E^*(X^{(p)})$ vanishes. It is very convenient that this should be the case, so that elements may be described as the (unique!) limit of Cauchy sequences. The following will suffice for our needs.

\begin{lem}\label{lem:Complete}
Let $\E$ be a spectrum with each $\pi_i(\E)$ a finitely-generated abelian group, and $\Gamma$ be a group admitting a finite $\underline{E}\Gamma$ and a torsion-free subgroup of finite index, then $\E^*(B\Gamma)$ is complete Hausdorff.
\end{lem}
\begin{proof}
We use a theorem of Jankovski \cite{Jankovski} (see Yosimura \cite[Theorem 6 vii)]{Yosimura} for the form of this theorem we will use). This theorem requires $B\Gamma$ to have finite skeleta, which follows from the fact that there is a finite-index $\Gamma' \leq \Gamma$ which is torsion-free and hence acts freely on $\underline{E}\Gamma$, which is a contractible finite $\Gamma'$-CW-complex. Thus $B\Gamma$ has a finite cover which has finite skeleta, so $B\Gamma$ has finite skeleta too.

Then \cite[Theorem 6 vii)]{Yosimura} says that as long as for each $\ell$ the map
$$\E^i(B\Gamma)\otimes \bQ \lra \prod_{k \leq \ell} \H^k(B\Gamma ; \pi_{i-k}(\E)\otimes \bQ),$$
given by the Chern--Dold character \cite[p.\ 172]{AdamsGuide} with target $\prod_k \H^k(- ; \pi_{i-k}(\E)\otimes \bQ)$ composed by projection to the factors having $k \leq \ell$, 
is surjective, then $\E^*(B\Gamma)$ is complete Hausdorff. 

To verify this condition, we consider the map $B\Gamma = E\Gamma \times_\Gamma \underline{E}\Gamma \to \underline{B}\Gamma = * \times_\Gamma \underline{E}\Gamma$ and form the induced commutative diagram
\begin{equation*}
\begin{tikzcd}
\E^i(\underline{B}\Gamma)\otimes \bQ \rar \dar & \prod_{k \leq \ell} \H^k(\underline{B}\Gamma ; \pi_{i-k}(\E)\otimes \bQ) \dar\\
\E^i(B\Gamma)\otimes \bQ \rar & \prod_{k \leq \ell} \H^k(B\Gamma ; \pi_{i-k}(\E)\otimes \bQ).
\end{tikzcd}
\end{equation*}
The right-hand vertical map is an isomorphism, as the fibres of $B\Gamma \to \underline{B}\Gamma$ are classifying spaces of finite groups and hence are rationally acyclic. As $\underline{B}\Gamma$ is a finite complex, the Chern--Dold character $\ch : \E^i(\underline{B}\Gamma)\otimes \bQ \to \prod_{k} \H^k(\underline{B}\Gamma ; \pi_{i-k}(\E)\otimes \bQ)$ is an isomorphism, so the top horizontal map is surjective. Thus the clockwise composition is surjective, and hence the bottom horizontal map is too.
\end{proof}

Combined with the discussion in the previous subsection we obtain:

\begin{cor}\label{cor:MgComplete}
If $\E$ is a spectrum with each $\pi_i(\E)$ a finitely-generated abelian group, then $\E^*(\mathcal{M}_g)$ for $g \geq 2$ and $\E^*(\mathcal{M}_{1}^1)$ are complete Hausdorff.\qed
\end{cor}

\subsection{Complex oriented cohomology theories}

Recall that a complex oriented cohomology theory means a ring spectrum $\E$ equipped with a choice of element $x = x^E \in \tilde{\E}^2(\mathbb{CP}^\infty)$ whose restriction to $\tilde{\E}^2(\mathbb{CP}^1) = \tilde{\E}^2(S^2) \cong \tilde{\E}^0(S^0)$ is the canonical element $1 \in \tilde{\E}^0(S^0)$. This yields a theory of Thom classes $u_V^E \in \tilde{\E}^{2\dim(V)}(\mathrm{Th}(V))$ for any complex vector bundle $\pi: V \to B$, as well as a theory of Chern classes $c_i^E(V) \in \E^{2i}(B)$. By construction the tautological bundle $\gamma \to \mathbb{CP}^\infty$ has $c_1^E(\gamma) = x$.

As $x \in \tilde{\E}^2(\mathbb{CP}^\infty) = F^2 \E^2(\mathbb{CP}^\infty)$ we have $x^{n+1} \in F^{2n+2} \E^{2n+2}(\mathbb{CP}^\infty)$, so this class vanishes on restriction to $\mathbb{CP}^n$. The induced map $\E^*[x]/(x^{n+1}) \to \E^*(\mathbb{CP}^n)$ can be checked to be an isomorphism by induction on $n$. It follows that the inverse system of $\{\E^*(\mathbb{CP}^n)\}_n$ is Mittag-Leffler, so taking inverse limits we find that the induced map
$$\E^*[[x]] \lra \E^*(\mathbb{CP}^\infty)$$
is an isomorphism. In particular $\E^*(\mathbb{CP}^\infty)$ is complete and Hausdorff. The analogous calculation goes through for finite products of copies of $\mathbb{CP}^\infty$. The map $\mu : \mathbb{CP}^\infty \times \mathbb{CP}^\infty \to \mathbb{CP}^\infty$ classifying the tensor product of the two tautological line bundles gives a canonical element
$$\mu^*(x) = F(x_1, x_2) \in \E^*(\mathbb{CP}^\infty \times \mathbb{CP}^\infty) = \E^*[[x_1, x_2]].$$
This is a formal group law, so has the form $x_1 + x_2 + \cdots$. Its $n$-series is
$$[n](x) := F(x, F(x, \ldots, F(x,x) \ldots)) \in \E^*[[x]],$$
with $n$ $x$'s. This has the form $[n](x) = n x + \cdots$, so is uniquely divisible by $x$ in $\E^*[[x]]$. Following Quillen, we define 
$$\emptyset_n(x) := \frac{[n](x)}{x} \in \E^*[[x]].$$

For $i : BC_n \to BS^1 \simeq \mathbb{CP}^\infty$ induced by considering $C_n \subset S^1$ as the $n$-th roots of unity, the Gysin sequence takes the form 
$$\cdots \lra \E^*[[x_n]] \overset{[n](x_n) \cdot -}\lra \E^*[[x_n]] \lra \E^*(BC_n) \lra \cdots$$
with $x_n := i^*(x)$. As long as $[n](x_n) \in \E^*[[x_n]]$ is not a zerodivisor\footnote{Which is the case in most examples of interest, e.g. integral cohomology, complex cobordism, complex $K$-theory, or whenever $\E^*$ is a local ring complete with respect to its maximal ideal \cite[Lemma 5.7]{HKR}.}, we therefore have $\E^*(BC_n) = \E^*[[x_n]]/([n](x_n))$.

\subsection{Complex cobordism}

The universal complex oriented cohomology theory is given by the complex cobordism spectrum $\MU$, with complex orientation given by the tautological map $x : \Sigma^\infty \mathbb{CP}^\infty \simeq \MU(1) \to \Sigma^{2} \MU$. Its associated formal group law is the universal formal group law, by a theorem of Quillen \cite[Theorem 6.5]{QuillenCob}, and in particular its coefficients are given by
$$\MU^* = \bZ[t_1, t_2, t_3,\ldots], \quad\quad |t_i| \in \MU^{-2i}.$$
It will be useful to know that we may take $t_1 = [\bP^1]$, and that the formal group law has the form
$$F(x_1, x_2) = x_1 + x_2  - [\bP^1]x_1 x_2 + \text{ higher order terms}$$
so that
\begin{equation}\label{eq:nSeriesMU}
[n](x) = nx - \binom{n}{2} [\bP^1] x^2 + \text{ higher order terms}.
\end{equation}

\subsubsection{Complex cobordism of $BC_n$}

In particular $\MU^*$ is a domain, so $\MU^*[[x]]$ is too, and hence $[n](x)$ is not a zerodivisor, giving $\MU^*(BC_n) = \MU^*[[x_n]]/([n](x_n))$. We record some further useful facts about it. 

\begin{thm}[Landweber \cite{Landweber}]\label{thm:MUCnFlat}
$\MU^*(BC_n)$ is a flat $\MU^*$-module (so in particular is $\bZ$-torsion-free).\qed
\end{thm}

\begin{lem}\label{lem:AnnOfx}
If $a \in \MU^*(BC_n)$ has $x_n \cdot a =0$, then $a = A \cdot \emptyset_n(x_n)$ for some $A \in \MU^{|a|}$.
\end{lem}
\begin{proof}
Representing $a$ by a series $f(x_n) \in \MU^*[[x_n]]$, the relation $x_n \cdot a =0$ means that $x_n \cdot f(x_n)$ is divisible by $[n](x_n) = x_n \cdot \emptyset_n(x_n)$ in $\MU^*[[x_n]]$, and so dividing by $x_n$  it shows that 
$$a = g(x_n) \cdot \emptyset_n(x_n)  \in \MU^*[[x_n]]/([n](x_n)) = \MU^*(BC_n)$$
for some $g(x_n) = \sum_{i=0}^\infty g_i x_n^i \in \MU^*[[x_n]]$.

We want to exploit that $x_n \cdot \emptyset_n(x_n)  =0 \in \MU^*(BC_n)$. As $\MU^*(BC_n)$ is complete and Hausdorff by Lemma \ref{lem:Complete}, the element $\left(\sum_{i=0}^\infty g_i x_n^i \right) \cdot \emptyset_n(x_n)$ agrees with the unique limit as $j \to \infty$ of the Cauchy sequence
$$\left(\sum_{i=0}^j g_i x_n^i \right) \cdot \emptyset_n(x_n) = g_0 \cdot \emptyset_n(x_n),$$
which is constant and so has limit $g_0  \cdot\emptyset_n(x_n)$. Thus $a$ has the claimed form.
\end{proof}

We will often use Quillen's theorem \cite[Proposition 4.2]{QuillenCob} that the Becker--Gottlieb transfer for the universal cyclic cover $\pi : EC_n \to BC_n$ is given by 
\begin{equation}\label{eq:QuillenTrf}
\mathrm{trf}_\pi^*(1) = \emptyset_n(x_n) \in \MU^*(BC_n).
\end{equation}
More generally:
\begin{lem}\label{lem:QuillenTrfLem}
If $p: BC_m \to BC_n$ is the finite covering space associated to the subgroup $C_m \leq C_n$ then
$$\trf_p^*(f(x_m)) = f(x_n) \cdot \emptyset_{n/m}([m](x_n)) \in \MU^*(BC_n).$$
\end{lem}
\begin{proof}
As $p^*(x_n)=x_m$ we have $f(x_m) = p^*(f(x_n))$, so the claim follows from the projection formula, \eqref{eq:QuillenTrf}, and the fact that $p$ is pulled back from $\pi : E C_{n/m} \to BC_{n/m}$ along $q : BC_n \to BC_{n/m}$, which satisfies $q^*(x_{n/m}) = [m](x_n)$.
\end{proof}

By naturality these formulas hold in any complex-oriented cohomology theory.

\subsubsection{Landweber--Novikov operations}\label{sec:LNops}

Let us record some limited facts about cohomology operations in $\MU$-theory: for much more see \cite[Part I]{AdamsBlue}. The universal Chern classes $c_i^{MU} \in \MU^{2i}(BU)$ correspond, under the Thom isomorphism, to certain cohomology operations $s_i \in \MU^{2i}(\MU)$. These are (some of the) Landweber--Novikov operations. If $\pi : V \to B$ is a $d$-dimensional (virtual) complex vector bundle with Thom class $u_V^{MU} \in \MU^{2d}(\mathrm{Th}(V))$, then this operation acts as $s_i(u_V^{MU}) = c_i^{MU}(V) \cdot u_V^{MU}$ by construction.

Taking $V$ to be the external sum of $n$ copies of the universal line bundle $\gamma \to \bC\bP^\infty$, it follows that $s_i(x_1 \cdots x_n) = \sigma_i(x_1, \ldots, x_n)\cdot x_1 \cdots x_n$, and restricting to the diagonal gives $s_i(x^n) = \binom{n}{i} x^{n+i} \in \MU^*(\bC\bP^\infty)$.

\subsubsection{Rationalised complex cobordism}\label{sec:RationalMU}

The spectrum $\MU_\bQ = \MU \otimes \H\bQ$ obtained by rationalising $\MU$ is complex-oriented via the rationalisation map. The Chern--Dold character
$$\mathrm{ch} : \MU_\bQ^*(-) \lra \H^*(- ; \MU^* \otimes \bQ)$$
is an isomorphism of multiplicative cohomology theories. The target of this isomorphism has another complex orientation coming from the additive formal group law on rational cohomology (tensored with $\MU^*$). In particular this gives two Thom classes $u_\gamma^{MU}, \mathrm{ch}^{-1}(u_\gamma^{H}) \in \MU_\bQ^*(\mathrm{Th}(\gamma \to \bC\bP^\infty))$, whose comparison is described as in \cite[Part II \S 9]{AdamsBlue} by
$$\mathrm{ch}^{-1}(u_\gamma^H) = \left(\sum_{n=0}^\infty \frac{[\bP^n]}{n+1} c_1^{MU}(\gamma)^{n} \right) \cdot u_\gamma^{MU}$$
and so, pulling back to the zero section, whose first Chern classes are related by
$$\mathrm{ch}^{-1}(c_1^H(\gamma)) = \sum_{n=0}^\infty \frac{[\bP^n]}{n+1} c_1^{MU}(\gamma)^{n+1}.$$

\begin{rem}\label{rem:RatIsCompleteHausdorff}
Although $\pi_*(\MU_\bQ)$ is not finitely-generated so Corollary \ref{cor:MgComplete} cannot be applied, it follows using the Chern--Dold character and the finite $\bQ$-cohomological dimension of $\mathcal{M}_g$ and $\mathcal{M}_1^1$ that $\MU_\bQ^*(\mathcal{M}_g)$ and $\MU_\bQ^*(\mathcal{M}_1^1)$ are complete Hausdorff. The same goes for any cohomology theory which is rational.
\end{rem}

\subsection{Complex $K$-theory}

For the complex $K$-theory spectrum $\K$ there is a selection of complex orientations available, and several choices are used in the literature. We take the choice $x := 1 - [\gamma^\vee] \in \tilde{\K}^0(\mathbb{CP}^\infty)$, where $\gamma$ denotes the tautological line bundle. By writing this element in $\K^0$ rather than $\K^2$ we are implicitly using the Bott periodicity isomorphism: furthermore, we are implicitly saying that we take $\beta := 1 - [\gamma^\vee] \in \tilde{\K}^0(\mathbb{CP}^1) = \tilde{\K}^{-2}(S^0)$ as our choice of Bott class. Then $\K^* = \bZ[\beta^{\pm 1}]$. As Mumford--Morita--Miller classes always have even degree, we will consistently use Bott periodicity to move all discussions into $\K^0$. With these conventions the formal group law is 
$$c_1^K(L_1 \otimes L_2) = c_1^K(L_1) + c_1^K(L_2) - c_1^K(L_1) \cdot c_1^K(L_2),$$
and so the $n$-series is
$$[n](x) = 1-(1-x)^n.$$
Again, as $\K^*$ is a domain we have $\K^0(BC_n) = \bZ[[x_n]]/(1-(1-x_n)^n)$.

We will sometimes use the Atiyah(--Segal) completion theorem \cite{AtiyahBG}: for a finite group $G$, the ring $\K^0(BG)$ is the completion of the complex representation ring $R(G)$ at its augmentation ideal $I$.

\subsection{Mumford--Morita--Miller classes}

If $\pi : E \to B$ is a surface bundle whose vertical tangent bundle $T_\pi$ is equipped with an orientation, then we may consider $T_\pi$ as a complex line bundle, classified by a map $\ell : E \to \mathbb{CP}^\infty$. The fibrewise Pontrjagin--Thom construction gives a map of spectra
$$c : \Sigma^\infty B_+ \lra \mathrm{Th}(-T_\pi \to E).$$
In a complex-oriented cohomology theory $\E$, we define fibre integration as the composition
$$\pi_!^E(-) : \E^*(E) \underset{\sim}{\overset{- \cdot u_{-T_\pi}}\lra} \E^{*-2}(\mathrm{Th}(-T_\pi \to E)) \overset{c^*}\lra \E^{*-2}(B),$$
where the first map denotes the Thom isomorphism. Using this, we define the $\E$-theory Mumford--Morita--Miller classes by
$$\kappa_i^E(\pi) := \pi_!^E(c_1^{E}(T_\pi)^{i+1}) \in \E^{2i}(B).$$
As $c_1^E(T_\pi) = \ell^*(x)$, these classes may be also be described as the pullback of the class $x^{i+1} \cdot u_{-\gamma} $ along the map of spectra
$$\Sigma^\infty B_+ \overset{c}\lra \mathrm{Th}(-T_\pi \to E) \overset{\mathrm{Th}(\ell)}\lra \mathrm{Th}(-\gamma \to \mathbb{CP}^\infty) =: \mathrm{MTSO}(2).$$
Applied to the universal family $\pi : \mathcal{C}_g \to \mathcal{M}_g$ we obtain universal classes 
$$\kappa_i^E \in \E^{2i}(\mathcal{M}_g), \quad i \geq -1.$$

\begin{rem}
One obvious difference between studying Mumford--Morita--Miller classes in ordinary cohomology versus general complex-oriented cohomology theories is that the class 
$$\kappa^{E}_{-1}(\pi) = \pi_!^{E}(1) \in \E^{-2}(B)$$
need not be zero. In fact, under Quillen's geometric interpretation \cite{QuillenCob} of complex cobordism, $\kappa^{MU}_{-1}(\pi) \in \MU^{-2}(B)$ is the class of the complex oriented map $\pi : E \to B$ itself.
\end{rem}

Composition with the ``zero-section'' map $i : \mathrm{Th}(-T_\pi \to E) \to \mathrm{Th}(-T_\pi \oplus T_\pi \to E) = \Sigma^\infty E_+$ gives (by definition \cite{BG}) the Becker--Gottlieb transfer 
$$\trf_\pi : \Sigma^\infty B_+ \overset{c}\lra \mathrm{Th}(-T_\pi \to E) \overset{i}\lra \mathrm{Th}(-T_\pi \oplus T_\pi \to E) = \Sigma^\infty E_+.$$
As $i^*(-) = (-) \cdot c_1^{E}(T_\pi) \cdot u_{-T_\pi}$, we obtain the identity $\pi_!^E(- \cdot c_1(T_\pi) ) = \trf_\pi^*(-)$, so for $i \geq 0$ we also have the description 
$$\kappa_i^E(\pi) = \trf_\pi^*(c_1^{E}(T_\pi)^{i}).$$

\begin{lem}\label{lem:FiltEstimate}
Restricting along $\iota: \{*\} \to \mathcal{M}_g$ we have
$$\iota^*\kappa^E_{-1} = (1-g)[\mathbb{P}^1]_E \in \E^{-2} \quad\quad \iota^*\kappa^E_{0} = 2-2g \in \E^0 \quad\quad \iota^*\kappa^E_j=0 \in \E^{2j} \text{ for $j>0$}.$$
More precisely, we have
\begin{align*}
\kappa^E_{-1} - (1-g)[\mathbb{P}^1]_E &\in F^1 \E^{-2}(\mathcal{M}_g)\\
\kappa^E_{0} - (2-2g) &\in F^1 \E^{0}(\mathcal{M}_g)\\
\kappa^E_{j}  &\in F^{2j} \E^{2j}(\mathcal{M}_g) \text{ for $j>0$}.
\end{align*}
\end{lem}
\begin{proof}
We may establish this universally in the theory $\E=\MU$. Restricted to a point, we are forming the pushforwards along $\pi : \Sigma_g \to \{*\}$ of powers of the class $c_1^{MU}(T\Sigma_g) \in \MU^2(\Sigma_g)$. Tautologically $\pi_!(1) = [\Sigma_g] \in \MU^{-2} = MU_2$ is the cobordism class of the oriented surface of genus $g$, and by computing the first Chern number this is cobordant to $(1-g)$ copies of the projective line $\mathbb{P}^1$. The number $\pi_!(c_1^{MU}(T\Sigma_g)) \in \MU^0=\MU_0=\bZ$ can be determined by applying $\MU \to \H$, so it is $\int_{\Sigma_g} c_1^H(T\Sigma_g)=2-2g$. For $j>0$ we have $\MU^{2j} = \MU_{-2j}=0$, so necessarily $\iota^*\kappa_j^{MU}=0$.%

The class $x \in \MU^2(\mathbb{CP}^\infty)$ vanishes when restricted to the 1-skeleton, so we have $c_1^{MU}(T_\pi) = \ell^*(x) \in F^2 \MU^2(\mathcal{C}_g)$, and so $c_1^{MU}(T_\pi)^{j+1} \in F^{2j+2} \MU^{2j+2}(\mathcal{C}_g)$. If $\mathcal{M}_g^{(p)}$ is a $p$-skeleton then $\mathcal{C}_g\vert_{\mathcal{M}_g^{(p)}}$ is a surface bundle over it so it has dimension $\leq p+2$, and hence $c_1^{MU}(T_\pi)^{j+1}$ vanishes on it if $p+2 < 2j+2$. But then $\kappa^{MU}_j\vert_{\mathcal{M}_g^{(p)}}$ vanishes if $p+2 < 2j+2$, so $\kappa^{MU}_j$ lies in $F^{2j} \MU^{2j}(\mathcal{M}_g)$.
\end{proof}

For a complex-oriented cohomology theory $\E$ for which $\E^*(\mathcal{M}_g)$ is complete Hausdorff (e.g.\ if the homotopy groups of $\E$ are finitely-generated, by Corollary \ref{cor:MgComplete}, or if $\E$ is rational as in Remark \ref{rem:RatIsCompleteHausdorff}) we let the \emph{tautological ring} $R\E^*(\mathcal{M}_g)$ denote the complete subring generated by $\kappa^E_{-1}, \kappa^E_{0}, \kappa^E_{1}, \ldots$.

\subsubsection{Cohomology operations}

Supposing that $\E^*(\mathcal{M}_g)$ is complete Hausdorff, so that $R\E^*(\mathcal{M}_g)$ is defined, we have the following:

\begin{lem}
Stable $\E$-theory cohomology operations preserve $R\E^*(\mathcal{M}_g)$.
\end{lem}
\begin{proof}
If $\psi$ is a stable $\E$-theory cohomology operation then
$$\psi(\kappa^E_i) = c^* \mathrm{Th}(\ell)^*(\psi(x^{i+1} \cdot u_{-\gamma}^E)),$$
and $\psi(x^{i+1} \cdot u_{-\gamma} ) =  f_{\psi, i}(x) \cdot u_{-\gamma}^E  \in \E^*(\mathrm{MTSO}(2)) =   \E^{*+2}(\mathbb{CP}^\infty) \cdot u_{-\gamma}^E$ for a unique power series $f_{\psi, i}(x) = \sum_{j=0}^\infty a_j \cdot x^j \in \E^*(\mathbb{CP}^\infty) = \E^*[[x]]$. Hence
$$\psi(\kappa^E_i) = \sum_{j=0}^\infty a_j \cdot \kappa^E_{j-1} \in \E^*(\mathcal{M}_g),$$
where we use the completeness of $\E^*(\mathcal{M}_g)$ and the filtration estimates in Lemma~\ref{lem:FiltEstimate} to interpret the right-hand side as a convergent Cauchy sequence.
\end{proof}

In $\MU$-theory, the action of the Landweber--Novikov operations has the following simple description.

\begin{lem}\label{lem:LNOps}
We have 
\begin{enumerate}[(i)]
\item $s_{i+1}(\kappa_{-1}^{MU}) = (-1)^{i+1} \kappa_i^{MU} \in \MU^{2i}(\mathcal{M}_g)$,

\item $s_j(\kappa_n^{MU}) = \binom{n}{j} \kappa_{n+j}^{MU} \in \MU^{2n+2j}(\mathcal{M}_g)$ for $n \geq 0$.

\end{enumerate}
\end{lem}
\begin{proof}
For (i), in $\MU^*(\mathrm{MTSO}(2)) = \MU^*(\mathbb{CP}^\infty) \cdot u_{-\gamma}$ we have
$$s_{i+1}(u_{-\gamma}) = c_{i+1}^{MU}(-\gamma) \cdot u_{-\gamma} = (-1)^{i+1} x^{i+1} \cdot u_{-\gamma}$$
as $c^{MU}(-\gamma) = \tfrac{1}{c^{MU}(\gamma)} = \tfrac{1}{1+x} \in \MU^*(\mathbb{CP}^\infty)$.

For (ii), we use that for $n \geq 0$ we have $\kappa_n^{MU} = \trf_\pi^*(c_1^{MU}(T_\pi)^n)$, that the stable cohomology operation $s_j$ commutes with $\trf_\pi^*$, and the formula $s_j(c_1^{MU}(T_\pi)^n) = \binom{n}{j} c_1^{MU}(T_\pi)^{n+j}$ from Section \ref{sec:LNops}.
\end{proof}

We may view the Chern--Dold character $\mathrm{ch} : \MU^*(-) \to \H^*(- ; \MU^* \otimes \bQ)$ as a cohomology operation between two different cohomology theories. When the target is equipped with the additive formal group law, the discussion in Section \ref{sec:RationalMU} explains how the theories of Chern classes and fibre integration in these two theories interact. Writing $g(z) := \sum_{n=0}^\infty \frac{[\bP^n]}{n+1} z^{n+1} \in\MU_\bQ^*[[z]]$, so that $\mathrm{ch}(x^{MU}) = g^{-1}(x^H) \in \H^*(\bC\bP^\infty ; \MU^* \otimes \bQ)$, it follows from that discussion that on $\mathrm{MTSO}(2)$ we have
$$\mathrm{ch}((x^{MU})^{i+1} \cdot u_{-\gamma}^{MU}) = g^{-1}(x^H)^{i} \cdot x^H \cdot u^H_{-\gamma},$$
and therefore
\begin{equation}\label{eq:ChOnKappas}
\mathrm{ch}(\kappa_i^{MU}) = \pi_!^H(g^{-1}(c_1^H(T_\pi))^{i} \cdot c_1^H(T_\pi)).
\end{equation}
In particular, $\mathrm{ch}(\kappa_i^{MU})$ is expressed as a (finite, as $\H^*(\mathcal{M}_g;\bQ)$ vanishes in high enough degrees) linear combination of $\kappa^H_i, \kappa^H_{i+1}, \kappa^H_{i+2}, \ldots$ with coefficients in $\MU^* \otimes \bQ$ (and coefficient of $\kappa_i^H$ equal to 1). From this it follows that the Chern--Dold character induces an isomorphism
\begin{equation}\label{eq:MURatTaut}
R\MU^*_\bQ(\mathcal{M}_g) \cong R\H^*(\mathcal{M}_g;\bQ) \otimes_\bZ \MU^*.
\end{equation}

\subsection{Hodge bundles}

If $\Sigma$ is a Riemann surface then there are associated vector spaces $H_q := H^0(\Sigma ; (T^\vee\Sigma)^{\otimes q})$ of holomorphic $q$-differentials. When $g(\Sigma) > 1$ these vanish for $q < 0$, so in this case their dimensions $h_q := \dim H_q$ may be calculated by Riemann--Roch, to be
$$h_q = \begin{cases}
0 & q < 0 \\
1 & q=0 \\
g & q=1 \\
(2q-1)(g-1) & q > 1.
\end{cases}$$
As their dimensions are constant, they assemble to vector bundles $\mathcal{H}_q \to \mathcal{M}_g$. The bundle $\mathcal{H}_0$ is the trivial 1-dimensional vector bundle. By Grothendieck--Riemann--Roch and Serre duality, if $\pi : \mathcal{C}_g \to \mathcal{M}_g$ is the universal family of Riemann surfaces then
$$\pi_!^K((T_\pi^\vee)^{\otimes q}) = \mathcal{H}_q - \mathcal{H}_{1-q}^\vee \in \K^0(\mathcal{M}_g).$$

In particular, using that $c_1^K(T_\pi) := 1 - T_\pi^\vee$ the first few $K$-theoretic Mumford--Morita--Miller classes are
\begin{align*}
\kappa^K_{-1} = \pi_!^K(1) &= 1 - \mathcal{H}_1^\vee\\
\kappa^K_{0} = \pi_!^K(1- T_\pi^\vee) &= 2 - \mathcal{H}_1 - \mathcal{H}_1^\vee\\
\kappa^K_{1} = \pi_!^K((1- T_\pi^\vee)^{\otimes 2}) &= 3 - 2\mathcal{H}_1 - \mathcal{H}_1^\vee  + \mathcal{H}_2. 
\end{align*}
Notable here is that $\mathcal{H}_2 \cong T^\vee \mathcal{M}_g$, so $\kappa_{-1}^K-2\kappa_0^K + \kappa_1^K$ is the class of the cotangent bundle of $\mathcal{M}_g$. Continuing in this way we see that for each $q \geq 1$, $\mathcal{H}_q$ is explicitly expressible in terms of $\kappa_{-1}^K, \kappa_0^K, \ldots, \kappa^K_{q-1}$.

\section{The variety of $\E^*(\mathcal{M}_g)\otimes \bQ$}

For a group $G$ let $\mathcal{A}(G)$ denote the category whose objects are the finite abelian subgroups of $G$ of prime power order, and whose morphisms are those homomorphisms induced by conjugation in $G$. A generalisation of Artin's theorem due to Hopkins--Kuhn--Ravenel \cite[Theorem A]{HKR} states that for a finite group $G$ and a complex-oriented cohomology theory $\E$ the natural map
\begin{equation}\label{eq:hkr}
\E^*(BG) \otimes \bQ \lra \lim_{A \in \mathcal{A}(G)} \E^*(BA)\otimes \bQ
\end{equation}
is an isomorphism (in fact, it is sufficient to invert $\vert G \vert$ rather than rationalising). By combining a version of this theorem with Quillen's work on varieties of group cohomology rings \cite{QuillenSpec}, we obtain a version applicable to the mapping class group.

\begin{thm}\label{thm:ModNil}
For $g \geq 2$ the map
$$\E^*(B\Gamma_g) \otimes \bQ \lra \lim_{A \in \mathcal{A}(\Gamma_g)} \E^*(BA)\otimes \bQ$$
is surjective with nilpotent kernel. The same holds for $\Gamma_{1}^1$.
\end{thm}

This is an instance of the following more general result about the complex oriented cohomology of classifying spaces of groups which admit a finite $\underline{E}\Gamma$ and a torsion-free finite-index subgroup: recall from Section \ref{sec:MgBackground} that $\Gamma_g$ (and $\Gamma_1^1 = \SL(2,\bZ)$) has these properties.

\begin{thm}
Let $\Gamma$ be a group admitting a finite $\underline{E}\Gamma$ and a torsion-free finite-index subgroup. 
Let $\E$ be a complex oriented cohomology theory. Then the map
$$\E^*(B\Gamma) \otimes \bQ \lra \lim_{A \in \mathcal{A}(\Gamma)} \E^*(BA)\otimes \bQ$$
is surjective with nilpotent kernel.%
\end{thm}

Note that if $\Gamma$ has a finite $\underline{E}\Gamma$ then it has finitely-many conjugacy classes of finite subgroups, so $\mathcal{A}(\Gamma)$ is (equivalent to) a finite category: hence it does not matter in which order we rationalise and take the limit on the right hand side.

\begin{proof}
Let $\Gamma' \leq \Gamma$ be a torsion-free finite-index subgroup, which we may assume to be normal. Then $\underline{E}\Gamma$ is an $E\Gamma'$, and we let $G := \Gamma/\Gamma'$. Then we may take $EG \times \underline{E}\Gamma$ with the diagonal $\Gamma$-action as an $E\Gamma$, so that $B\Gamma = EG \times_G (\underline{E}\Gamma/\Gamma')$. Let us write $Y := \underline{E}\Gamma/\Gamma'$, which is a $B\Gamma'$ and a finite CW-complex. We now apply \cite[Theorem A]{HKR}, which says that the natural maps
$$\E^*(EG \times_G Y) \lra \lim_{A \in \mathcal{A}(G)} \E^*(EG \times_A Y) \lra \int_{A \in \mathcal{A}(G)} \E^*(BA \times Y^A)$$
are both isomorphisms after inverting $\vert G \vert$. (This uses the observation of \cite[Remark 3.5 (1)]{HKR} to restrict to abelian groups of prime power order.)

Consider $\E^*(BA \times Y^A)$ equipped with the Leray--Serre filtration with respect to the map $BA \times Y^A \to Y^A$. The edge homomorphism for the associated Atiyah--Hirzebruch--Serre spectral sequence has the form
$$\E^*(BA \times Y^A) \lra \H^0(Y^A ; \E^*(BA)),$$
where we may think of the target as the set $\mathrm{Map}(Y^A, \E^*(BA))$ of locally constant functions. This edge homomorphism is split surjective, by choosing a splitting of $Y^A \to \pi_0(Y^A)$. Its kernel consists of elements which have Serre filtration $\geq 1$, so their $d$-th powers have Serre filtration $\geq d$ and hence vanish as long as $d > \mathrm{dim}(Y^A)$: in particular this kernel is nilpotent. Using that the category $\mathcal{A}(G)$ is finite, it follows that
$$\int_{A \in \mathcal{A}(G)} \E^*(BA \times Y^A) \lra \int_{A \in \mathcal{A}(G)} \mathrm{Map}(Y^A, \E^*(BA))$$
is (split) surjective with nilpotent kernel. Finally, the argument of \cite[Lemma 15.4]{QuillenSpec}, translated into the language of ends, shows that the natural map
\begin{equation*}
\int_{A \in \mathcal{A}(G)} \mathrm{Map}(Y^A, \E^*(BA)) \lra \int_{A \in \mathcal{A}(\Gamma)} \mathrm{Map}(\underline{E}\Gamma^A, \E^*(BA))
\end{equation*}
is an isomorphism. As each $A \in \mathcal{A}(\Gamma)$ is in particular a finite group, the fixed points $\underline{E}\Gamma^A$ are contractible and in particular connected so the right hand side is $\int_{A \in \mathcal{A}(\Gamma)} \E^*(BA)$, which is precisely the limit in the statement of the theorem.
\end{proof}

\begin{rem}
An extension of the Hopkins--Kuhn--Ravenel character theory to groups with a finite $\underline{E}\Gamma$ has been given by L{\"u}ck--Patchkoria--Schwede \cite{LPS}, and it would be interesting to investigate its consequences for $\Gamma_g$.
\end{rem}

There is a similar result to the above that is worth recording here. For a group $G$, let $\mathcal{F}(G)$  denote the category whose objects are the finite subgroups of $G$, and whose morphisms are those homomorphisms induced by conjugation in $G$.

\begin{thm}\label{thm:FinSubgpDetNilp}
Let $\Gamma$ be a group admitting a finite $\underline{E}\Gamma$ and a torsion-free finite-index subgroup. Let $\E$ be a multiplicative cohomology theory. Then the natural map
$$\E^*(B\Gamma)\lra \prod_{F \in \mathcal{F}(\Gamma)} \E^*(BF)$$
has nilpotent kernel.
\end{thm}
\begin{proof}
As in the previous proof, we have $B\Gamma = EG \times_G (\underline{E}\Gamma/\Gamma')$ with $Y := \underline{E}\Gamma/\Gamma'$ a finite $G$-CW-complex. The Leray spectral sequence for the map $EG \times_G Y \to G\backslash Y$ has edge homomorphism
$$\E^*(B\Gamma) \lra H^0(G\backslash Y ; \mathcal{E}^*_G),$$
where $\mathcal{E}^*_G$ is the sheaf on $G\backslash Y$ whose stalk at $[y]$ is the $\E$-cohomology of the stabiliser of $y \in Y$. This edge homomorphism has nilpotent kernel by the argument of \cite[Proposition 3.2]{QuillenSpec} (cf.\ Remark 3.4 of that paper). But if an element is in the kernel of the map in the statement of the theorem, then it is also in the kernel of this edge homomorphism.
\end{proof}

\section{Finite group actions}\label{sec:FinGpAct}

If $\rho : G \to \Diff^+(\Sigma)$ describes an action of a group $G$ on a surface $\Sigma$ by orientation-preserving diffeomorphisms, the Borel construction gives an associated oriented surface bundle
$$EG \times_G \Sigma \lra BG,$$
and so associated classes $\kappa_i^{MU}(\rho) \in \MU^{2i}(BG)$.

Following Kawazumi--Uemura \cite{KawazumiUemura}, when $G$ is finite we will give a formula for $\kappa_i^{MU}(\rho)$ in terms of fixed-point data, as long as $i \geq 0$. (We don't know whether $\kappa_{-1}^{MU}(\rho)$ can have such a description in general, but for cyclic groups we give one in Corollary \ref{cor:KappaMinusOneCyclic}.) We will deduce it from a more general result whose flavour is that of the localisation theorem in equivariant cohomology. In the following, for a subgroup $H \leq G$ we write $\mathrm{Ind}_{H}^G : \MU^*(BH) \to \MU^*(BG)$ for the effect of the Becker--Gottlieb transfer map associated to the finite cover $BH \to BG$.

\begin{thm}\label{thm:LocalisationTheorem}
Let $G$ be finite and $\rho$ be a faithful action. Let $S \subset \Sigma$ be the (finite) set of points having nontrivial isotropy, and write $G_p \leq G$ for the isotropy group at a point $p \in \Sigma$ (which acts faithfully on $T_p\Sigma$ so is cyclic). 
\begin{enumerate}[(i)]
\item If $y \in \MU^i_G(\Sigma) := \MU^i(EG \times_G \Sigma)$ has degree $i > 2$ then
$$\pi_!^{MU}(y) = \sum_{[p] \in S/G} \mathrm{Ind}_{G_p}^G\left(\frac{y\vert_p}{c_1^{MU}(T_p \Sigma)}\right) \in \MU^{i-2}(BG).$$
Part of the statement is that the restriction $y\vert_p \in \MU^i_{G_p}(\{p\})$ is divisible by $c_1^{MU}(T_p \Sigma)$, and the fraction denotes any choice of quotient.

\item If $y \in \MU^2_G(\Sigma)$ then
$$\pi_!^{MU}(y) = C \cdot \mathrm{Ind}_{\{e\}}^G(1) + \sum_{[p] \in S/G} \mathrm{Ind}_{G_p}^G\left(\frac{y\vert_p}{c_1^{MU}(T_p \Sigma)}\right) \in \MU^{0}(BG)$$
where $C \in \mathbb{Z}$ is determined by requiring that the two sides be equal after restricting along $\{*\} \to BG$. (This will depend on the choices of quotients $\frac{y\vert_p}{c_1^{MU}(T_p \Sigma)}$.)
\end{enumerate}
\end{thm}
\begin{proof}
Let us first deal with the case $S=\emptyset$, i.e.\ $G$ acts freely on $\Sigma$. Then $\MU^i_G(\Sigma) = \MU^i(\Sigma/G)$. If $i>2$ then this group vanishes, so $y=0$ and hence (i) holds trivially. If $i=2$ then Poincar{\'e} duality $[\Sigma/G]_{MU} \frown - : \MU^2(\Sigma/G) \overset{\sim}\to \MU_0(\Sigma/G)=\bZ$ identifies $y$ with the integer $d := [\Sigma/G]_{MU} \frown y$. If $\iota : \{*\} \to \Sigma/G$ is a point, then $\iota^{MU}_!(1) \in \MU^2(\Sigma/G)$ is a generator and $y = d \cdot \iota^{MU}_!(1)$. Then $q : \{*\} \overset{\iota}\to \Sigma/G \overset{\pi}\to BG$ are complex-oriented maps, and $\pi^{MU}_!\iota^{MU}_!(1) = q^{MU}_!(1) = \mathrm{Ind}_{\{e\}}^G(1)$, so $\pi^{MU}_!(y) = d \cdot \mathrm{Ind}_{\{e\}}^G(1)$ as required.

Suppose now that $S$ is nonempty. We may choose a $G$-invariant Riemannian metric on $\Sigma$, and hence cover it $G$-equivariantly by $\Sigma - S$ and a small open disc bundle of $T\Sigma\vert_S$. Using that this disc bundle is ($G$-equivariantly) homotopy equivalent to $S$, and that the intersection is ($G$-equivariantly) homotopy equivalent to the unit sphere bundle $S(T\Sigma\vert_S)$, we develop the following commutative  diagram.
\begin{equation*}
\begin{tikzcd}%
\MU^{i-1}_G(S(T\Sigma\vert_S)) \arrow[rd, "\partial"] \arrow[d, "\delta"] \\
\bigoplus_{[p] \in S/G} \MU^{i-2}_{G_p}(\{p\}) \rar{i^{MU}_!} \arrow[rdd, swap, "\bigoplus c_1^{MU}(T_p\Sigma) \cdot -"]& \MU^i_G(\Sigma) \rar \dar & \MU^i_G(\Sigma - S) \dar\\
  & \MU^i_G(S) \rar \arrow[d, equals] & \MU^i_G(S(T\Sigma\vert_S)) \arrow[d, equals]\\
 & \bigoplus_{[p] \in S/G} \MU^i_{G_p}(\{p\}) & \bigoplus_{[p] \in S/G} \MU^i_{G_p}(S(T_p\Sigma))
\end{tikzcd}
\end{equation*}
The row is part of the long exact sequence for the pair $(\Sigma, \Sigma - S)$ together with excision and the Thom isomorphism. The square and the map $\partial$ can be viewed as part of the Mayer--Vietoris sequence for the open cover of $\Sigma$ by $\Sigma - S$ and the small open disc bundle of $T\Sigma\vert_S$.

As $G$ acts freely on $\Sigma-S$ we have $\MU^i_G(\Sigma-S) = \MU^i((\Sigma-S)/G)$ and $(\Sigma-S)/G$ has dimension 1 as $S \neq \emptyset$. In particular this vanishes if $i \geq 2$. Similarly, as $G_p$ acts freely on $S(T_p \Sigma)$ we have $\MU^i_{G_p}(S(T_p\Sigma)) =  \MU^i(S(T_p\Sigma)/G_p)$ and $S(T_p\Sigma)/G_p = S^1$ so this vanishes if $i \geq 2$.

Suppose then that we are given $y \in \MU^i_G(\Sigma)$ with $i \geq 2$. It restricts to 0 in $\MU^i_G(\Sigma - S)$ so lifts along $i_!^{MU}$, and hence its restriction $\bigoplus_{[p]} y\vert_p$ to $\MU^i_G(S) = \bigoplus_{[p] \in S/G} \MU^i_{G_p}(\{p\})$ admits a lift $\bigoplus_{[p]} y\vert_p / c_1^{MU}(T_p\Sigma)$ to $\bigoplus_{[p] \in S/G} \MU^{i-2}_{G_p}(\{p\})$. Thus the class
\begin{equation}\label{eq:ModifiedY}
y - \sum_{[p]} i^{MU}_! \left( \frac{y\vert_p }{ c_1^{MU}(T_p\Sigma)}\right) \in \MU^i_G(\Sigma)
\end{equation}
restricts to 0 in both $\MU^i_G(\Sigma - S)$ and $\MU^i_G(S)$, so by exactness of the Mayer--Vietoris sequence it is in the image of the map $\partial$. If $i > 2$ then $\MU^{i-1}_G(S(T\Sigma\vert_S))=0$ so \eqref{eq:ModifiedY} is zero and applying $\pi^{MU}_!$ to it gives (i).

If $i=2$ then $\MU^{1}_G(S(T\Sigma\vert_S)) = \bigoplus_{[p]} \MU^1(S(T_p\Sigma)/G_p)$ is non-zero. The diagram shows that there are classes $\epsilon_p \in \MU^0_{G_p}(\{p\})$ so that \eqref{eq:ModifiedY} is $\sum_{[p]} i^{MU}_!(\epsilon_p)$ and so
 \begin{equation}\label{eq:ModifiedYFormula}
y = \sum_{[p]} i^{MU}_!\left( \frac{y\vert_p }{ c_1^{MU}(T_p\Sigma)} + \epsilon_p\right) \in \MU^2_G(\Sigma).
\end{equation}
As the $i^{MU}_!(\epsilon_p)$ come from $\partial$ they vanish in $\MU^i_G(S)$, and so $c_1^{MU}(T_p\Sigma) \cdot \epsilon_p  = 0$. Thus the element $\frac{y\vert_p }{ c_1^{MU}(T_p\Sigma)} + \epsilon_p$ is just another choice of a quotient of $y\vert_p$ by $c_1^{MU}(T_p\Sigma)$ in the ring 
$$\MU^*_{G_p}(\{p\}) = \MU^*[[c_1^{MU}(T_p\Sigma)]]/([|G_p|](c_1^{MU}(T_p\Sigma))).$$

By Lemma \ref{lem:AnnOfx} the non-uniqueness involved in dividing by $c_1^{MU}(T_p\Sigma)$ in the ring $\MU^*_{G_p}(\{p\})$ consists only of $\MU^*$-multiples of $\emptyset_{|G_p|}(c_1^{MU}(T_p\Sigma)) = \mathrm{Ind}_{\{e\}}^{G_p}(1)$.
Thus $\epsilon_p  = E_p \cdot \mathrm{Ind}_{\{e\}}^{G_p}(1)$ for some $E_p \in \bZ = \MU^0$. Applying $\pi_!^{MU}$ to \eqref{eq:ModifiedYFormula} gives
\begin{align*}
\pi_!^{MU}(y) &= \sum_{[p]} \pi^{MU}_! i^{MU}_! \left( \frac{y\vert_p }{ c_1^{MU}(T_p\Sigma)} + E_p \cdot \mathrm{Ind}_{\{e\}}^{G_p}(1)\right) \\
&= \sum_{[p]} \mathrm{Ind}_{G_p}^G \left(\frac{y\vert_p }{ c_1^{MU}(T_p\Sigma)}\right) + E_p \cdot \mathrm{Ind}_{\{e\}}^{G}(1)\\
&= \left(\sum_{[p]} E_p\right) \cdot \mathrm{Ind}_{\{e\}}^{G}(1) + \sum_{[p]} \mathrm{Ind}_{G_p}^G \left(\frac{y\vert_p }{ c_1^{MU}(T_p\Sigma)}\right) \in \MU^0(BG).
\end{align*}
This gives a formula of the required kind, with $C := \sum_{[p]} E_p$. This integer can be determined by restricting to a point.
\end{proof}

We now apply this general result about fibre-integration to the generalised Mumford--Morita--Miller classes. On applying the map $\MU^*(-) \to \H^*(-;\bZ)$ it recovers the formula of Kawazumi--Uemura \cite{KawazumiUemura}.

\begin{cor}\label{cor:FinGpAct}
With the same notation as above, we have
$$\kappa_i^{MU}(\rho) = \sum_{[p] \in S/G} \mathrm{Ind}_{G_p}^G(c_1^{MU}(T_p \Sigma)^i) \in \MU^{2i}(BG) \quad \text{ for }i \geq 1$$
and
$$\kappa_0^{MU}(\rho) = \frac{\chi(\Sigma) - \vert S \vert}{\vert G \vert} \cdot \mathrm{Ind}_{\{e\}}^G(1) + \sum_{[p] \in S/G} \mathrm{Ind}_{G_p}^G(1) \in \MU^{0}(BG).$$
The same identities then hold in any complex-oriented cohomology theory.
\end{cor}
\begin{proof}
For $\kappa_i^{MU}(\rho)$ with $i \geq 1$ we apply Theorem \ref{thm:LocalisationTheorem} (i) with $y = c_1^{MU}(T\Sigma)^{i+1}$. We may take $c_1^{MU}(T\Sigma)^{i}$ as a choice of the quotient $\frac{c_1^{MU}(T\Sigma)^{i+1}}{c_1^{MU}(T\Sigma)}$, giving the claimed formula.

For $\kappa_0^{MU}(\rho)$  we apply Theorem \ref{thm:LocalisationTheorem} (ii) with $y = c_1^{MU}(T\Sigma)$. We take $1$ as a choice of the quotient $\frac{c_1^{MU}(T\Sigma)}{c_1^{MU}(T\Sigma)}$, giving
$$\kappa_0^{MU}(\rho) = C \cdot \mathrm{Ind}_{\{e\}}^G(1) + \sum_{[p] \in S/G} \mathrm{Ind}_{G_p}^G(1) \in \MU^{0}(BG)$$
for a $C \in \mathbb{Z}$ to be determined by restriction to a point. Such a restriction gives
$$\chi(\Sigma) = C \cdot |G| + \sum_{[p] \in S/G} |G/G_p|$$
and hence $C = (\chi(\Sigma) - |S|)/|G|$.
\end{proof}

By a theorem of Kerckhoff \cite{Kerckhoff}, for $g \geq 2$ a homomorphism $G \to \pi_0 \Diff^+(\Sigma)$ from a finite group lifts (in fact uniquely up to conjugacy) to a homomorphism to $\Diff^+(\Sigma)$, and in particular has fixed-point data associated to it, using which we may apply the above to describe its Mumford--Morita--Miller classes.
\subsection{Cyclic group actions}

Let us apply the formulas of Corollary \ref{cor:FinGpAct} to the case of cyclic group actions on surfaces.

\begin{cor}\label{cor:DistinguishCylic}
For two faithful actions $\rho$ and $\rho'$ of $C_n$ on a surface $\Sigma$ the following are equivalent
\begin{enumerate}[(i)]

\item $\kappa_{-1}^{MU}(\rho) = \kappa_{-1}^{MU}(\rho') \in \MU^{-2}(BC_n)$,

\item $\kappa_i^{MU}(\rho) = \kappa_i^{MU}(\rho') \in \MU^{2i}(BC_n)$ for all $i \geq -1$,

\item $\kappa_i^{K}(\rho) = \kappa_i^{K}(\rho') \in \K^{0}(BC_n)$ for all $i \geq 1$.
\end{enumerate}
If $n=p^k$ is a prime power, then the above are also equivalent to
\begin{enumerate}[(i)]
\setcounter{enumi}{3}
\item $\rho$ and $\rho'$ are conjugate.

\end{enumerate}
\end{cor}

In the proof we will use that cyclic group actions on surfaces can be described combinatorially, and classified, according to the following scheme of Nielsen \cite{Nielsen}: we follow the exposition in Symonds \cite{Symonds}. An action of the cyclic group $G:= C_n = \mu_n \subseteq S^1$ of $n$-th roots of unity on $\Sigma$, where the generator acts via a diffeomorphism $\phi$, is determined up to conjugation by the unordered collection of elements $\beta_1/\alpha_1, \ldots, \beta_q/\alpha_q \in \bQ/\bZ$ where
\begin{enumerate}[(a)]
\item the orbits $S/G$ (of the set of points $S \subset \Sigma$ having nontrivial stabiliser) has representatives $x_1, \ldots, x_q \in \Sigma$,

\item $\alpha_i$ is the order of the stabiliser of $x_i$, so $G_{x_i} = \mu_{\alpha_i}$,

\item $\phi^{n/\alpha_i}$ acts on $T_{x_i} \Sigma$ by multiplication by an $\alpha_i$th root of unity $\zeta \in \bC^\times$, and $\beta_i$ is such that $\zeta^{\beta_i} = e^{2\pi i/\alpha_i}$.
\end{enumerate}

\begin{proof}[Proof of Corollary \ref{cor:DistinguishCylic}]
If $n=p_1^{e_1} \cdots p_r^{e_r}$ is the factorisation into distinct prime powers, so that $C_n = C_{p_1^{e_1}} \times \cdots \times C_{p_r^{e_r}}$, then the map
$$BC_{p_1^{e_1}} \vee \cdots \vee BC_{p_r^{e_r}} \lra BC_n,$$
induced by the canonical inclusion of groups on each factor, is a stable homotopy equivalence. It follows that the classes $\kappa^E_i(\rho) \in \E^*(BC_n)$ are completely determined by their restrictions to the Sylow subgroups $C_{p_i^{e_i}}$ of $C_n$. Thus to prove the equivalence of (i)---(iii) we may as well suppose that $C_n = C_{p^k}$ is a cyclic $p$-group, and also show these properties are equivalent to (iv). 

Clearly (iv) implies (i). Using Landweber--Novikov operations in the form of $s_{i+1}(\kappa^{MU}_{-1}) = (-1)^{i+1} \kappa_{i}^{MU}$ from Lemma \ref{lem:LNOps}, it follows that (i) implies (ii). Clearly (ii) implies (iii).  We shall show that (iii) implies (iv) using Nielsen's classification described above.

Recall that our convention for the complex orientation of $\K$ is such that $x_n := c_1^K(L_n) = 1-[L_n^\vee] \in \K^0(BC_n)$, where $L_n$ is the tautological representation of $C_n = \mu_n$ on $\bC$ in which the generator acts as $e^{2\pi i/n}$. If $k$ divides $n$ then $\mu_{n/k}$ is a subgroup of $\mu_n$, and the restriction of $x_n$ is $x_{n/k}$.

In Nielsen's classification $\alpha_i$ divides $n=p^k$ so is a $p$-th power, and is the order of a nontrivial stabiliser so is divisible by $p$. Furthermore $\beta_i$ is coprime to $\alpha_i$, so we can write $A_i \alpha_i + B_i \beta_i = 1$ for integers $A_i$ and $B_i$ and furthermore $B_i$ and $\beta_i$ are not divisible by $p$. Given $\alpha_i$, the data $[B_i] = [\beta_i]^{-1} \in \bZ/\alpha_i$ is equivalent to the data $[\beta_i] \in \mathbb{Z}/\alpha_i$, and hence to the data $\beta_i/\alpha_i \in \bQ/\bZ$.

By (c) there is an isomorphism $(T_{x_i}\Sigma)^{\otimes \beta_i} \cong L_{\alpha_i}$ of representations of $G_{x_i} = \mu_{\alpha_i}$, and so an isomorphism 
$$T_{x_i}\Sigma \cong (T_{x_i}\Sigma)^{\otimes(A_i \alpha_i + B_i \beta_i)} \cong ((T_{x_i}\Sigma)^{\otimes \beta_i})^{\otimes B_i} \cong L_{\alpha_i}^{\otimes B_i}$$
of representations of $G_{x_i} = \mu_{\alpha_i}$. It follows that
$$c_1^{K}(T_{x_i} \Sigma) = [B_i](x_{\alpha_i}) \in \K^0(BG_{x_i})$$
where $x_{\alpha_i} = c_1^{K}(L_{\alpha_i})$ as above.  Thus for all $j \geq 1$, using Lemma \ref{lem:QuillenTrfLem} we have
$$\kappa_j^{K}(\rho) = \sum_{i=1}^q \mathrm{Ind}_{G_{x_i}}^G([B_i](x_{\alpha_i})^j) = \sum_{i=1}^q [B_i](x_{n})^j \cdot \frac{[n](x_n)}{[\alpha_i](x_n)} \in \K^0(BC_n).$$

Now $\K^0(BC_n)$ is the completion of the complex representation ring $R(C_n) = \bZ[L_n]/(L_n^{n}-1)$ at its augmentation ideal $I=(L_n-1) = (1-L_n^\vee)$, and as $n=p^k$ it follows from \cite[III Proposition 1.1]{AtiyahTall} that the $I$-adic topology on $I$ coincides with the $p$-adic topology. Thus
$$\K^0(BC_n) = \mathbb{Z} \times_{\mathbb{Z}_p} \mathbb{Z}_p[x_n]/(1-(1-x_n)^{p^k}),$$
so that there is a homomorphism $\Phi_k$ to $\mathbb{Z}_p[\xi_{p^k}]$, the extension of the $p$-adics obtained by adjoining a primitive $p^k$-th root of unity, given by by $1-x_n \mapsto \xi_{p^k}$. This homomorphism sends $[B](x_n)$ to $1-\xi_{p^k}^B$, so applying it to $\kappa^K_j(\rho)=\kappa^K_j(\rho')$ gives identities
\begin{equation}\label{eq:PS}
\sum_{\substack{i \text{ such that } \\ p^k \mid \alpha_i}} (1-\xi_{p^k}^{B_i})^j = \sum_{\substack{i \text{ such that } \\ p^k \mid \alpha'_i}} (1-\xi_{p^k}^{B'_i})^j \in \mathbb{Z}_p[\xi_{p^k}] \quad \text{ for all $j \geq 1$}.
\end{equation}
Here we have used that $\Phi_k$ annihilates $\frac{[n](x_n)}{[\alpha_i](x_n)}$ if $p^k  \nmid \alpha_i$. Of course $p^k \mid \alpha_i$ just means $p^k = \alpha_i$, but we write the argument this way so that below it will easily generalise.

If $p_j(z_1, \dots, z_N) = \sum_{i=1}^N z_i^j$ denotes the power sum, there is an identity
$$\exp\left(-\sum_{j=1}^\infty \tfrac{1}{j} p_j(z_1, \ldots, z_N) T^j\right) = \prod_{i=1}^N (1-z_i T) \in \bQ[z_1, \ldots, z_N][[T]].$$
By functoriality, the same holds in $R[[T]]$ for any $\bQ$-algebra $R$ and sequence of elements $z_1, \ldots, z_N \in R$. Note that extending a sequence by taking $z_{N+1}=0$ gives exactly the same identity.

This applies to the above considered as lying in the $\bQ$-algebra given by the field of fractions $\mathbb{Q}_p(\xi_{p^k})$, where the identity \eqref{eq:PS} of power sums shows that
$$\prod_{\substack{i \text{ such that } \\ p^k = \alpha_i}} (1-(1-\xi_{p^k}^{B_{i}}) T) = \prod_{\substack{i \text{ such that } \\ p^k = \alpha'_i}} (1-(1-\xi_{p^k}^{B'_{i}}) T) \in \mathbb{Q}_p(\xi_{p^k})[[T]].$$
Both sides define elements of the subring $\mathbb{Q}(\xi_{p^k})[T]$, so they are equal there. As $p$ does not divide $B_i$ or $B'_i$, both $1-\xi_{p^k}^{B_{i}}$ and $1-\xi_{p^k}^{B'_{i}}$ are nonzero. By unique factorisation of polynomials over a field it follows that the degrees of these polynomials are equal, so
$$\#\{i \text{ such that } p^k = \alpha_i\} = \#\{i \text{ such that } p^k = \alpha'_i\},$$
and we may order the fixed points with stabilisers having order $p^k$ so that
$$1-\xi_{p^k}^{B_{i}} = 1-\xi_{p^k}^{B'_{i}} \in \bZ[\xi_{p^k}] \subset \mathbb{Q}_p(\xi_{p^k}).$$
This is only possible if $B_{i} \equiv B'_{i} \mod p^k$, so the fixed point data of the two actions at points of maximal isotropy are the same.

We generalise the above as follows. For each $m = 1,2,\ldots, k$ there is a homomorphism $\Phi_m : \K^0(BC_n) \to \mathbb{Z}_p[\xi_{p^m}]$ given by $1-x_n \mapsto \xi_{p^m}$. This satisfies
$$\Phi_m\left(\frac{[n](x_n)}{[\alpha_i](x_n)}\right) = \Phi_m(\emptyset_{n/\alpha_i}([\alpha_i](x_n))) = \begin{cases}
0 & p^m \nmid \alpha_i\\
\frac{n}{\alpha_i} & p^m \mid \alpha_i
\end{cases}$$
and so applying it to $\kappa^K_j(\rho)=\kappa^K_j(\rho')$ gives 
\begin{equation*}
\sum_{\substack{i \text{ such that } \\ p^m \mid \alpha_i}} \tfrac{n}{\alpha_i} (1-\xi_{p^m}^{B_i})^j = \sum_{\substack{i \text{ such that } \\ p^m \mid \alpha'_i}} \tfrac{n}{\alpha'_i} (1-\xi_{p^m}^{B'_i})^j \in \mathbb{Z}_p[\xi_{p^m}] \quad \text{ for all $j \geq 1$}.
\end{equation*}
By downwards induction on $m$, with the base case $m=k$ proved above, the terms with $p^{m+1} \mid \alpha_i$ and $p^{m+1} \mid \alpha'_i$ already match up, leaving
\begin{equation*}
\sum_{\substack{i \text{ such that } \\ p^m = \alpha_i}} (1-\xi_{p^m}^{B_i})^j = \sum_{\substack{i \text{ such that } \\ p^m = \alpha'_i}} (1-\xi_{p^m}^{B'_i})^j \in \mathbb{Z}_p[\xi_{p^m}] \quad \text{ for all $j \geq 1$}.
\end{equation*}
Proceeding with the power-sum method as above, we find that the two actions have the same number of points with isotropy of order $p^m$, and that these may be ordered so that $B_{i} \equiv B'_{i} \mod p^m$.
\end{proof}

\begin{rem}
The restriction to prime-power order is necessary: it is not hard to e.g.\ construct nonconjugate actions of  $C_6$ on a surface which become conjugate when restricted to $C_2$ or to $C_3$, and hence have equal Mumford--Morita--Miller classes.
\end{rem}

The following gives a formula for $\kappa_{-1}^{MU}(\rho)$ for cyclic group actions: it would be interesting to try to extend it to products of cyclic groups and hence, via \eqref{eq:hkr}, obtain a formula for general finite group actions.

\begin{cor}\label{cor:KappaMinusOneCyclic}
Let $\rho$ be a faithful orientation-preserving action of a cyclic group $G=C_n$ on the oriented surface $\Sigma_g$. Then in $\MU^*(BC_n) \otimes \bQ$ we have
$$\kappa^{MU}_{-1}(\rho) = A[\mathbb{P}^1] \cdot \emptyset_n(x_n) +  \sum_{i=1}^q \emptyset_{n/\alpha_i}([\alpha_i](x_n)) \cdot \frac{\emptyset_{\beta_i}([B_i](x_{n})) - \beta_i/\alpha_i\cdot \emptyset_{\alpha_i}(x_{n})}{x_n}$$
where the quotients are formed in $\MU^*[[x_n]] \otimes \bQ$, and $A \in \bQ$ is determined by 
$$1-g = nA - \sum_{i=1}^q n/\alpha_i \cdot (\binom{\beta_i}{2} B_i - \tfrac{\beta_i}{\alpha_i} \binom{\alpha_i}{2}).$$
\end{cor}
Note that although this formula is in the rationalisation $\MU^*(BC_n) \otimes \bQ$, as $\MU^*(BC_n)$ is torsion-free (Theorem \ref{thm:MUCnFlat}) it determines $\kappa^{MU}_{-1}(\rho) \in \MU^*(BC_n)$.

\begin{proof}
We apply Theorem \ref{thm:LocalisationTheorem} with $y = \pi^*(x_n) \in \MU^2_{C_n}(\Sigma_g)$, giving
$$x_n \cdot \kappa^{MU}_{-1}(\rho) = x_n \cdot \pi^{MU}_!(1) = \pi^{MU}_!(y) = C \cdot \mathrm{Ind}_{\{e\}}^{C_n}(1) + \sum_{i=1}^q \mathrm{Ind}_{C_{\alpha_i}}^{C_n}\left(\frac{x_{\alpha_i}}{[B_i](x_{\alpha_i})}\right).$$
We have $x_{\alpha_i} = [\beta_i]([B_i](x_{\alpha_i}))$ mod $[\alpha_i](x_{\alpha_i})$, so we can choose to form the quotient
$$\frac{x_{\alpha_i}}{[B_i](x_{\alpha_i})} = \emptyset_{\beta_i}([B_i](x_{\alpha_i}))$$
Restricting to a point this gives $0 = C \cdot |C_n| + \sum_{i=1}^q |C_n/C_{\alpha_i}| \beta_i$, so $C = - \sum_{i=1}^q \beta_i/\alpha_i$ (which will indeed be an integer). With some rearranging, in $\MU^*(BC_n) \otimes \bQ$ this gives
$$x_n \cdot \kappa^{MU}_{-1}(\rho) = \sum_{i=1}^q \emptyset_{n/\alpha_i}([\alpha_i](x_n)) \cdot \left( \emptyset_{\beta_i}([B_i](x_{n})) - \beta_i/\alpha_i\cdot \emptyset_{\alpha_i}(x_{n}) \right).$$
By Lemma \ref{lem:AnnOfx}, in $\MU^*(BC_n) \otimes \bQ$ we may divide by $x_n$ uniquely up to adding on $\MU^*$-multiples of $\emptyset_n(x_n)$, so
$$\kappa^{MU}_{-1}(\rho) = A[\mathbb{P}^1] \cdot \emptyset_n(x_n) +  \sum_{i=1}^q \emptyset_{n/\alpha_i}([\alpha_i](x_n)) \cdot \frac{\emptyset_{\beta_i}([B_i](x_{n})) - \beta_i/\alpha_i\cdot \emptyset_{\alpha_i}(x_{n})}{x_n}$$
for some $A[\mathbb{P}^1] \in \MU^{-2} \otimes \bQ$, which depends on how we choose to form the quotients on the right-hand terms in the ring $\MU^*(BC_n)$. If we form the quotients in $\MU^*[[x_n]]$, where they are unique, then using $[n](y) = ny - \binom{n}{2}[\mathbb{P}^1]y^2 + \cdots$ we have
$$[\Sigma_g] = nA[\mathbb{P}^1] - \sum_{i=1}^q n/\alpha_i \cdot (\binom{\beta_i}{2} B_i - \tfrac{\beta_i}{\alpha_i} \binom{\alpha_i}{2})[\mathbb{P}^1]$$
and $[\Sigma_g] = (1-g)[\mathbb{P}^1]$ so $A$ is determined by
$$1-g = nA - \sum_{i=1}^q n/\alpha_i \cdot (\binom{\beta_i}{2} B_i - \tfrac{\beta_i}{\alpha_i} \binom{\alpha_i}{2}).$$
\end{proof}

\subsection{Relation to the $G$-signature}
If $\Sigma$ is an oriented surface, then the intersection form endows $H^1(\Sigma;\bR)$ with a nondegenerate symplectic form $\lambda(-,-)$. Choosing an auxiliary inner product $\langle-,-\rangle$, we obtain an operator $A$ on this space defined by $\lambda(x,y) = \langle x, Ay \rangle$. As $\lambda$ is symplectic, $A^* = -A$. Thus we may form $J := A/\sqrt{AA^*}$, taking the positive square root of the positive self-adjoint operator $AA^*$, and this satisfies $J^2 = -I$, so $(H^1(\Sigma;\bR), J)$ is a complex vector space. 

If $\rho : G \to \mathrm{Diff}^+(\Sigma)$ is an orientation-preserving action of a finite group $G$ on $\Sigma$ then we may take the inner product $\langle-,-\rangle$ to be $G$-invariant (by averaging), and hence obtain a complex $G$-representation $(H^1(\Sigma;\bR), J)$. The \emph{$G$-signature} of this action is the virtual complex representation
$$\mathrm{sign}(\rho) := (H^1(\Sigma;\bR), -J) - (H^1(\Sigma;\bR), J) \in R(G).$$

After passing from $R(G)$ to $R(G)^\wedge_I = \K^0(BG)$, the $G$-signature may be expressed in terms of the $K$-theory Mumford--Morita--Miller classes $\kappa^K_i$. Namely, using that $\kappa_{-1}^K = \pi^K_!(1) = 1 - [\mathcal{H}_1^\vee]$ and checking carefully that $\cH_1^\vee$ corresponds to $(H^1(\Sigma;\bR), J)$ we have
 $$\mathrm{sign}(\rho) = \kappa_{-1}^K - \overline{\kappa_{-1}^K} \in R(G)^\wedge_I = \K^0(BG).$$

\begin{rem}
This gives another approach to (i) $\Rightarrow$ (iv) in Corollary \ref{cor:DistinguishCylic}, as follows. If $\kappa_{-1}^{MU}(\rho) = \kappa_{-1}^{MU}(\rho')$ then $\kappa_{-1}^{K}(\rho) = \kappa_{-1}^{K}(\rho')$ and so $\mathrm{sign}(\rho) = \mathrm{sign}(\rho') \in R(G)^\wedge_I$. Under the assumption that $G$ is a cyclic $p$-group $R(G)$ injects into its $I$-adic completion, so in fact $\mathrm{sign}(\rho) = \mathrm{sign}(\rho') \in R(G)$. But then a theorem of Edmonds--Ewing \cite[Theorem 3.2 and its proof]{EdmondsEwing} implies (using the classification of conjugacy classes of periodic diffeomorphisms in terms of fixed-point data, so not very different from the argument we gave) that $\rho$ is conjugate to $\rho'$.
\end{rem}

\subsection{The class $\kappa_0$}

For a surface bundle $\pi : E \to B$ the class $\kappa_0^{MU}(\pi)$ is equal to $\trf_\pi^*(1)$, the Becker--Gottlieb transfer applied to the unit. As the unit $1 \in \MU^0(E)$ lifts to the unit $1 \in \pi^0_S(E)$ in stable cohomotopy, there is a class
$$\kappa_0^S(\pi) = \trf_\pi^*(1) \in \pi^0_S(B)$$
which maps to $\kappa_0^{MU}(\pi)$ under the unit map.

When $B=BG$ for $G$ finite, the Segal conjecture \cite{Carlsson} identifies $\pi^0_S(BG)$ with the completion $A(G)^{\wedge}_I$ of the Burnside ring of $G$. Additivity of the Becker--Gottlieb transfer and the decomposition used in the proof of Theorem \ref{thm:LocalisationTheorem} (ii) give the following (see \cite[Section 5.1]{GMT} for the case of cyclic groups):

\begin{thm}
Let $G$ be finite and $\rho$ be a faithful action on a surface $\Sigma$, and keep the notation of Section \ref{sec:FinGpAct}. Then
\begin{equation}\label{eq:Burnside}
\pushQED{\qed} 
\kappa_0^{S}(\rho) = \frac{\chi(\Sigma) - \vert S \vert}{\vert G \vert} \cdot [G/e] + \sum_{[p] \in S/G} [G/G_p] \in \pi^{0}_S(BG) = A(G)^{\wedge}_I. \qedhere
\popQED
\end{equation}
\end{thm}

\begin{cor}\mbox{}
\begin{enumerate}[(i)]
\item The class $\kappa_0^{S}(\rho) \in \pi_S^0(BG)$ is determined by the numbers $\chi(\Sigma^H)$ for every subgroup $H \leq G$. 

\item The class $\kappa_0^{MU}(\rho)\otimes \bQ \in \MU^0(BG) \otimes \bQ$ is determined by the numbers $\chi(\Sigma^H)$ for every abelian subgroup $H \leq G$ of prime power order.

\item The class $\kappa_0^{K}(\rho)\otimes \bQ \in \K^0(BG) \otimes \bQ$ is determined by the numbers $\chi(\Sigma^H)$ for every cyclic subgroup $H \leq G$ of prime power order.
\end{enumerate}
\end{cor}

\begin{proof}
The formula
$$\mathcal{K} := \frac{\chi(\Sigma) - \vert S \vert}{\vert G \vert} \cdot [G/e] + \sum_{[p] \in S/G} [G/G_p]$$
gives an element of the Burnside ring $A(G)$. For a subgroup $H \leq G$ we have $\vert \mathcal{K}^H \vert = \chi(\Sigma^H)$, and it is classical that an element  $\mathcal{K} \in A(G)$ is determined by the function $H \mapsto \vert \mathcal{K}^H \vert$. As $\mathcal{K}$ determines $\kappa_0^S(\rho)$, the first statement follows.

For the second statement, it follows from \eqref{eq:hkr} that $\kappa_0^{MU}(\rho)\otimes \bQ$ is determined by $\kappa_0^{MU}(\rho\vert_A)\otimes \bQ$ for all abelian subgroups $A \leq G$ of prime power order; to these we then apply the first case of this statement.

For the third statement, consider the image of $\mathcal{K}$ under the linearisation map $A(G) \to R(G)$ to the complex representation ring. By character-theory, elements here are detected by restriction to cyclic subgroups. On the other hand, as in the beginning of the proof of Corollary \ref{cor:DistinguishCylic} elements of $\K^0(BC_n)$ are determined by their restrictions to the prime-power cyclic subgroups.
\end{proof}

The class $\kappa_0^{H} \in \H^0(\mathcal{M}_g;\bZ)$ is just the integer $\chi(\Sigma_g)=2-2g$, but $\kappa_0^{S} \in \pi_S^0(\mathcal{M}_g)$ is not simply an integer, or even a scalar. Nonetheless, it is an algebraic integer in the sense that it satisfies a monic polynomial with integer coefficients, as we now explain. Recall that $\Gamma_g$ has finitely-many conjugacy classes of finite subgroups.

\begin{thm}
Let $p(z) \in \bZ[z]$ be the monic polynomial which has a root at $\chi(\Sigma_g^F)$ for every finite $p$-subgroup $F \leq \Gamma_g$. Then some power of $p(\kappa_0^{S}) \in \pi_S^0(\mathcal{M}_g)$ vanishes.
\end{thm}
\begin{proof}
By Theorem \ref{thm:FinSubgpDetNilp} it suffices to show that $p(\kappa_0^S(F)) \in \pi^0_S(BF)$ vanishes for every finite subgroup $F \leq \Gamma_g$, and by Kerckhoff's theorem these groups are realised by actual actions on surfaces. By the Segal conjecture $\pi^0_S(BF) = A(F)^\wedge_I$, the completion of the Burnside ring of $F$ at the augmentation ideal. The formula \eqref{eq:Burnside} describes a lift $\mathcal{K}_F \in A(F)$ of $\kappa_0^S(F)$ to the uncompleted Burnside ring. Now, Laitinen \cite[Proposition 1.10]{Laitinen} has shown that the kernel of the completion map $A(F) \to A(F)^\wedge_I \cong \pi^0_S(BF)$ coincides with the intersection of the kernels of the restriction maps $A(F) \to A(F_p)$ to all Sylow subgroups $F_p \leq F$. But for any $p$-subgroup $P \leq \Gamma_g$ the element $p(\mathcal{K}_P) \in A(P)$ satisfies
$$|p(\mathcal{K}_P)^{P}| = p(|\mathcal{K}_P^P|) = p(\chi(\Sigma_g^P))$$
which vanishes as $p$ has a root at $\chi(\Sigma_g^P)$ by design. Thus $p(\mathcal{K}_F) \in A(F)$ vanishes after restriction to any $p$-subgroup, so is in the kernel of the completion map and hence $p(\kappa_0^S(F)) =0 \in \pi^0_S(BF)$ for all finite subgroup $F \leq \Gamma_g$.
\end{proof}

\begin{ex}
In \cite{BroughtonLow} one may find, for $g \in \{2,3\}$, a list of all conjugacy classes of finite subgroups $G$ of $\Gamma_g$, as well as branching data for the associated branched cover $\Sigma_g \to \Sigma_g/G$. From this it is easy to compute $\chi(\Sigma_g^G)$ and so to compute polynomials annihilating $\kappa_0^{S}$ up to nilpotence in genus 2 and 3:
\begin{align*}
p_2(z) &= (z+2)(z-6)(z-2)(z-4)(z)(z-3)\\
p_3(z) &= (z+4)(z-8)(z-4)(z)(z-5)(z-2)(z-3).
\end{align*}
How can one determine the minimal polynomial of $\kappa_0^{E} \in \E^0(\mathcal{M}_g)$? What does it mean?
\end{ex}

\begin{ex}
In \cite[Table 5 (3.k)]{BroughtonLow}, there is described an action of $C_6$ on $\Sigma_3$ having branching data $(0;2,3,3,6)$. Thus
$$\kappa_0^{S}(C_6) = -2[\tfrac{C_6}{e}] + [\tfrac{C_6}{C_2}] + 2[\tfrac{C_6}{C_3}] + 1 \in \pi^0_S(BC_6),$$
and the right hand side describes an element $\mathcal{K} \in A(C_6)$. However $\mathcal{K}^{C_6}=1$, which is not a root of $p_3(z)$, so $p_3(\mathcal{K}) \neq 0 \in A(C_6)$. In fact
$$p_3(\mathcal{K}) =  -140\left([\tfrac{C_6}{C_3}]-2\right)\left([\tfrac{C_6}{C_2}]-3\right) \in A(C_6).$$
But this does vanish when restricted to $C_2, C_3 \leq C_6$, so $p_3(\kappa_0^S(C_6))=0 \in \pi^0_S(BC_6)$.
\end{ex}

\section{A genus 1 example: The universal elliptic curve}

The group $\Gamma_{1}^1 = \SL(2, \bZ)$ acts on the torus $T = \bR^2 / \bZ^2$, and the associated surface bundle
\begin{equation}\label{eq:UnivEll}
\pi : E\SL(2, \bZ) \times_{\SL(2, \bZ)} T \lra B\SL(2, \bZ)
\end{equation}
has a canonical section $s$ given by the fixed point $[0,0] \in T$ of the action. This surface bundle is a homotopy-theoretic model for the universal elliptic curve $\mathcal{C}_{1}^1 \to \mathcal{M}_{1}^1$.

The amalgamated free product decomposition $\SL(2, \bZ) = C_6 *_{C_2} C_4$ proves that the abelianisation map $a: \SL(2, \bZ) \to C_{12}$ induces an isomorphism on integral (co)homology, so is a stable homotopy equivalence and hence also induces an isomorphism on $\MU$-theory of the associated classifying spaces. Thus, by the calculation of the $\MU$-theory of cyclic groups, we have
$$\MU^*(B\SL(2, \bZ)) = \MU^*[[x]]/([12](x)).$$
Here $x := (Ba)^*(c_1^{MU}(L))$ where $L \to BC_{12}$ is the complex line bundle given by the action of $C_{12}$ on $\bC$ as the twelfth roots of unity. Using this amalgamated free product decomposition and Corollary \ref{cor:FinGpAct}, we are able to compute \emph{all} of the $\MU$-theory $\kappa$-classes of \eqref{eq:UnivEll}.

\begin{thm}\label{thm:elliptic}
Firstly $\kappa_i^{MU} = x^{i+1}\kappa_{-1}^{MU}$, and furthermore
the class $\kappa_{-1}^{MU}$
is 
$$\frac{1}{x} \left(1-\frac{[12](x)}{[3](x)} \cdot \left(\frac{1}{\tfrac{[12](x)}{[4](x)}-2}\right) + \frac{[12](x)}{[4](x)} + \left(\frac{\tfrac{[12](x)}{[2](x)}-2\tfrac{[12](x)}{[4](x)}}{3-\tfrac{[12](x)}{[6](x)}}\right) \right) + 4[\bP^1]\frac{[12](x)}{x}.$$
\end{thm}

This formula must be interpreted in $\MU^*[[x]]$, as follows: 
\begin{enumerate}[(i)]
\item
For $k$ dividing $12$ the expression $\tfrac{[12](x)}{[k](x)}$ is to be interpreted as the quotient in $\MU^*[[x]]$, which agrees with $\emptyset_{12/k}([k](x))$.

\item The denominators $\tfrac{[12](x)}{[4](x)}-2 = 1 + (\tfrac{[12](x)}{[4](x)}-3)$ and $3-\tfrac{[12](x)}{[6](x)} = 1+ (2-\tfrac{[12](x)}{[6](x)})$ both lie in $1+(x) \subset \MU^*[[x]]$, and so can be formally inverted. It is these formal inverses to which the formula refers.

\item The long expression in brackets is uniquely divisible by $x$ in $\MU^*[[x]]$.%
\end{enumerate}

We start with some basic identities in the complex cobordism of cyclic groups.

\begin{lem}\label{lem:powers}
Consider $\frac{[kl](x)}{[l](x)} = \emptyset_k([l](x)) \in \MU^0(BC_{kl})$. Then
$$\emptyset_k([l](x))^2 = k \cdot \emptyset_k([l](x)),$$
and so $(\emptyset_k([l](x)) - k)^2 =  - k \cdot (\emptyset_k([l](x)) - k)$. 
\end{lem}
\begin{proof}
Note that $\emptyset_k([l](x)) = k + ([l](x))$, and $[l](x) \cdot \emptyset_k([l](x))=0$.
\end{proof}

\begin{lem}\label{lem:units}
In $\MU^0(BC_n)$ we have $\emptyset_n(x) = \emptyset_n([k](x))$ for all $k$ coprime to $n$.
\end{lem}
\begin{proof}
If $k$ is coprime to $n$, we may write $Ak + Bn = 1$ for integers $A$ and $B$. Hence
$$x = [A]([k](x)) +_F [B]([n](x)) = [A]([k](x))$$
so $x \in ([k](x))$ (and of course $[k](x) \in (x)$ so $(x)=([k](x))$, though we will not need this). Thus $x \cdot \emptyset_n([k](x))=0$, as $\emptyset_n([k](x))$ is annihilated by multiplication by $[k](x)$, and so $\emptyset_n([k](x))$ is a multiple of $\emptyset_n(x)$. As $x$ annihilates $\emptyset_n(x)$ this multiple is just an integer, and setting $x=0$ shows that this integer is 1.
\end{proof}

\begin{lem}\label{lem:products}
If $n=kl$ with $k$ and $l$ coprime, then in $\MU^0(BC_n)$ we have 
$$\emptyset_k([l](x)) \cdot \emptyset_{l}([k](x)) = \emptyset_{n}(x).$$
More generally, if $k$ and $l$ divide $n$ then
$$\emptyset_k([\tfrac{n}{k}](x)) \cdot \emptyset_l([\tfrac{n}{l}](x)) = \gcd(k, l) \cdot \emptyset_{kl/\gcd(k,l)}([\tfrac{n}{kl /\gcd(k,l)}](x)).$$
\end{lem}
\begin{proof}
In the first case as $k$ is coprime to $l$, we may write $Ak + Bl = 1$ for integers $A$ and $B$. Hence
$$x = [A]([k](x)) +_F [B]([l](x)) \in ([k](x), [l](x)).$$
As $\emptyset_k([l](x)) \cdot \emptyset_{l}([k](x))$ is annihilated by multiplication by $[k](x)$ or $[l](x)$, it follows that it is also annihilated by $x$, so is a multiple of $\emptyset_n(x)$. As in the previous lemma, this multiple is 1.

In the second case consider the quotient map $\phi : C_n \to C_{kl /\gcd(k,l)}$. If $z \in \MU^2(BC_{kl /\gcd(k,l)})$ denotes the canonical generator then $\phi^*(z) = [\tfrac{n\cdot \gcd(k,l)}{kl}](x)$ and so
$$[\tfrac{n}{k}](x) = \phi^*([\tfrac{l}{\gcd(k,l)}](z)) \quad\text{and}\quad [\tfrac{n}{l}](x) = \phi^*([\tfrac{k}{\gcd(k,l)}](z)).$$
Thus the identity we wish to establish is
$$\phi^*\left(\emptyset_k([\tfrac{l}{\gcd(k,l)}](z)) \cdot \emptyset_l([\tfrac{k}{\gcd(k,l)}](z))\right) = \phi^*\left(\gcd(k,l) \cdot \emptyset_{kl/\gcd(k,l)}(z)\right),$$
so we may as well show that it holds before pulling back via $\phi$.

As $\gcd(\tfrac{k}{\gcd(k,l)}, \tfrac{l}{\gcd(k,l)})=1$ we have $z \in ([\tfrac{k}{\gcd(k,l)}](z), [\tfrac{l}{\gcd(k,l)}](z))$ and so $z$ annihilates the product $\emptyset_k([\tfrac{l}{\gcd(k,l)}](z)) \cdot \emptyset_l([\tfrac{k}{\gcd(k,l)}](z))$. As above, it follows that the product is a multiple of $\emptyset_{kl/\gcd(k,l)}(z)$, and we can determine which multiple by setting $x$ to zero: the multiple is $\gcd(k,l)$.
\end{proof}

\begin{rem}
Using these lemmas, the formula of Theorem \ref{thm:elliptic} may be expanded in $\MU^*(B\SL(2,\bZ)) \otimes \bZ[\tfrac{1}{6}]$ to
$$\kappa_{-1}^{MU} = \frac{1}{x} \left( 1 + \frac{1}{3} \frac{[12](x)}{[4](x)} + \frac{[12](x)}{[3](x)} - \frac{1}{3} \frac{[12](x)}{[2](x)} - \frac{1}{2} \frac{[12](x)}{x} \right)  + 4[\bP^1]\frac{[12](x)}{x}.$$
As $\MU^*(B\SL(2,\bZ))$ is torsion-free (Theorem \ref{thm:MUCnFlat}) this formula entails no loss of information.
\end{rem}

\begin{proof}[Proof of Theorem \ref{thm:elliptic}]
Taking the differential of the map which translates (using the group structure on the torus $T$) to the origin gives a map on tangent bundles $\phi : T(T) \to T_0 T$, and if $A \in \SL(2,\bZ)$ then this satisfies
$$\phi \circ D A = D_0 A \circ \phi : T(T) \lra T_0 T,$$
so is equivariant with respect to the $\SL(2,\bZ)$-action on either side. On the bundle $\pi$, this translates into an isomorphism of bundles
$$T_\pi E \cong \pi^*(s^*(T_\pi E)),$$
and as $s^*(T_\pi E)$ is isomorphic to $(Ba)^*(L)$ we have $c_1^{MU}(T_\pi E) = \pi^*(x)$. Thus
$$\kappa_i^{MU} = \pi^{MU}_!(c_1^{MU}(T_\pi E)^{i+1}) = \pi_!^{MU}(\pi^*(x^{i+1})) = x^{i+1} \cdot \pi_!^{MU}(1).$$
This establishes the first claim.

To compute the class $\kappa_{-1}^{MU} \in \MU^*[[x]]/([12](x))$, note that by Lemma \ref{lem:AnnOfx} the class $\kappa_0^{MU} = x \cdot \kappa_{-1}^{MU}$ determines $\kappa_{-1}^{MU}$ up to addition of a $\bZ$-multiple of $[\bP^1] \cdot \emptyset_{12}(x)$. So we will first determine $\kappa_0^{MU}$, and then find the required multiple of $[\bP^1] \cdot \emptyset_{12}(x)$. The same principle was used in Corollary \ref{cor:KappaMinusOneCyclic}, and in fact the argument is highly similar to the one given there.

The amalgamated free product decomposition $\SL(2, \bZ) = C_6 *_{C_2} C_4$ gives a Mayer--Vietoris sequence in $\MU$-theory, but as $\MU^*(BC_n)$ is concentrated in even degrees this splits into short exact sequences. In particular, the square
\begin{equation*}
\begin{tikzcd}
\MU^0(B\SL(2,\bZ)) \rar \dar & \MU^0(BC_4) \dar\\
\MU^0(BC_6) \rar & \MU^0(BC_2)
\end{tikzcd}
\end{equation*}
is cartesian. Define an element
$$\rho(x) := 1-\emptyset_4([3](x)) \cdot \left(\frac{1}{\emptyset_3([4](x))-2}\right)  + \emptyset_3([4](x)) + \left(\frac{\emptyset_6([2](x))-2\emptyset_3([4](x))}{3-\emptyset_2([6](x))}\right)$$
in $\MU^0(B\SL(2,\bZ))$. We shall show that $\kappa_0^{MU} = \rho(x)$, by showing that this is the case when restricted to $BC_6$ and $BC_4$. Supposing we have done so, note that $\rho(x)$ is divisible by $x$ in $\MU^*[[x]]$ with quotient $\tfrac{\rho(x)}{x}$. Then $x \cdot (\kappa_{-1}^{MU} - \tfrac{\rho(x)}{x}) = \kappa_0^{MU} - \rho(x) =0$, so
$$\kappa_{-1}^{MU} = \tfrac{\rho(x)}{x} + \lambda  [\bP^1] \cdot \emptyset_{12}(x)$$
for some $\lambda \in \bZ$. Finally, restricting to a point $* \in B\SL(2,\bZ)$ has the effect of setting $x$ to zero so, using that $\emptyset_n(x) = n - \binom{n}{2}[\bP^1]x + \cdots$ from \eqref{eq:nSeriesMU}, after a tedious calculation we find
$$\MU^{-2} \ni  [T] = -48 \cdot [\bP^1] + \lambda \cdot 12 \cdot [\bP^1],$$
and as the torus $T$ is nullbordant as a stably almost complex manifold (it has Chern number 0) it follows that $\lambda=4$.

We may use Corollary \ref{cor:FinGpAct} to compute $\kappa_0^{MU}$ restricted to $C_6$ and $C_4$. The $C_6$-action on the torus has one orbit of size 1, one orbit of size 3, and one orbit of size 2; the $C_4$-action on the torus has two orbits of size 1, and one orbit of size 2, giving:
\begin{align*}
\kappa_0^{MU}(C_6) &= 1 - \emptyset_6(x) + \emptyset_3([2](x)) + \emptyset_2([3](x))\\
\kappa_0^{MU}(C_4) &= 2 - \emptyset_4(x) + \emptyset_2([2](x)).
\end{align*}

\vspace{1ex}

\noindent \textbf{$\rho(x)$ restricted to $BC_4$.} Here $[4](x)=0$, so $[6](x)=[2](x)$ and $[3](x) = [-1](x)$ too, and we obtain
$$\rho(x)\vert_{C_4} = 4-\emptyset_4([-1](x))  + \left(\frac{\emptyset_6([2](x))-6}{3-\emptyset_2([2](x))}\right).$$
By Lemma \ref{lem:units} we have $\emptyset_4([-1](x)) = \emptyset_4(x)$. As elements of $\MU^*[[x]]$ we have
$$[2](x) \cdot \emptyset_6([2](x)) = [12](x) = [4](x) +_F [4](x) +_F [4](x) = 3 \cdot [4](x) + [4](x)^2 \cdot X$$
for some $X \in \MU^*[[x]]$, and so
$$\emptyset_6([2](x)) = 3 \cdot \frac{[4](x)}{[2](x)} + [4](x) \cdot \tfrac{[4](x)}{[2](x)} \cdot X.$$
Thus on $BC_4$ we have $\emptyset_6([2](x)) = 3 \cdot \emptyset_2([2](x))$, so
$$\rho(x)\vert_{C_4} = 4-\emptyset_4(x)  + 3 \cdot \left(\frac{\emptyset_2([2](x))-2}{1 -(\emptyset_2([2](x))-2)}\right).$$
Lemma \ref{lem:powers} shows that
$$3 \cdot \frac{\emptyset_2([2](x))-2}{1 -(\emptyset_2([2](x))-2)} = \emptyset_2([2](x))-2$$
so we obtain
$$\rho(x)\vert_{C_4} = 2-\emptyset_4(x)+\emptyset_2([2](x))$$
as required.

\vspace{1ex}

\noindent \textbf{$\rho(x)$ restricted to $BC_6$.} By an argument similar to the last part, here $\emptyset_6([2](x))$ becomes $2\cdot \emptyset_3([2](x))$ and $\emptyset_4([3](x))$ becomes $2 \cdot \emptyset_2([3](x))$. By Lemma \ref{lem:units} (pulled back along $BC_6 \to BC_3$) we have $\emptyset_3([4](x)) = \emptyset_3([2](x))$, so the last term vanishes. We have
$$\rho(x)\vert_{C_6} = 1-\emptyset_2([3](x)) \cdot \left(\frac{2}{\emptyset_3([2](x))-2}\right)  + \emptyset_3([2](x)).$$
Lemma \ref{lem:powers} shows that
$$\frac{2}{(\emptyset_3([2](x))-3)+1} = \emptyset_3([2](x))-1$$
so we obtain
$$\rho(x)\vert_{C_6} = 1-\emptyset_2([3](x)) \cdot \emptyset_3([2](x)) + \emptyset_2([3](x)) + \emptyset_3([2](x)).$$
and by Lemma \ref{lem:products} this is $1-\emptyset_6(x) + \emptyset_2([3](x)) + \emptyset_3([2](x))$, as required.
\end{proof}

\begin{rem}
The reader may wonder how the formula $\rho(x)$ was guessed. The reason for working with $\kappa_0^{MU}$, and the idea for the formulas of Lemmas \ref{lem:powers} -- \ref{lem:products}, is that $\kappa_0^{MU}(\pi) = \trf_\pi^*(1)$ is obtained by the Becker--Gottlieb transfer applied to the unit, and so has a canonical lift $\kappa_0^S(\pi)$ to stable cohomotopy.

When $B=BG$, for an action of a finite group $G$ on a surface $\Sigma$, we have the formula \eqref{eq:Burnside}
 for $\kappa_0^S(\pi) \in \pi^0_S(BG) = A(G)^{\wedge}_I$. Although $\SL(2,\bZ)$ is not a finite group, the map $Ba : B\SL(2,\bZ) \to BC_{12}$ is a stable homotopy equivalence so induces an isomorphism in stable cohomotopy. So to obtain a formula for $\kappa_0^{MU}(SL(2, \bZ))$ it is enough to produce an element of $A(C_{12})^{\wedge}_I$ restricting to both
\begin{align*}
[C_6 / C_6] - [C_6] + [C_6/C_2] + [C_6 / C_3] &\in A(C_6)^{\wedge}_I\\
2[C_4 / C_4] - [C_4] + [C_4/C_2] &\in A(C_4)^{\wedge}_I.
\end{align*}
This is conceptually rather easier, in part because Lemmas \ref{lem:powers} -- \ref{lem:products} become obvious identities between $C_n$-sets in the Burnside ring, and some experimentation leads to the formula $\rho(x)$. However, although these two elements actually lie in the Burnside ring itself, before completion, it is easy to see that there is no element of $A(C_{12})$ reducing to both if them: it is essential that we complete $A(C_{12})$.
\end{rem}

\section{A genus 2 example: $\bQ \otimes R\K^0(\mathcal{M}_2)/\sqrt{0}$}

In order to determine $\bQ \otimes R\K^0(\mathcal{M}_2)/\sqrt{0}$ it follows from Theorem \ref{thm:ModNil} that we must determine those expressions in the $\kappa_i^K$ which vanish on all abelian subgroups of $\Gamma_2$ of prime power order; by the Atiyah(-Segal) completion theorem \cite{AtiyahBG} and Artin's theorem, it suffices to check vanishing only on \emph{cyclic} groups of prime power order. Furthermore it suffices to check vanishing only on the maximal such subgroups. According to Broughton's classification, these are $C_2$ (given by a \emph{non}-hyperelliptic involution, item (2.b) in Broughton's list), $C_3$, $C_5$, and $C_8$. Evaluating the $K$-theory Mumford--Morita--Miller classes here gives a homomorphism
\begin{equation}\label{eq:KM2Map}
\bZ[\kappa^K_{-1}, \kappa^K_0, \kappa^K_1,\ldots] \lra R(C_2) \times_\bZ R(C_3) \times_\bZ R(C_5) \times_\bZ R(C_8),
\end{equation}
and so a continuous homomorphism 
\begin{equation*}%
\bZ[[\kappa^K_{-1}+1, \kappa^K_0+2, \kappa^K_1,\ldots]] \lra \K^0(\mathcal{M}_2)\lra R(C_2)^\wedge_I \times_\bZ R(C_3)^\wedge_I \times_\bZ R(C_5)^\wedge_I \times_\bZ R(C_8)^\wedge_I.
\end{equation*}
Thus $\bQ \otimes R\K^0(\mathcal{M}_2)/\sqrt{0}$ is the rationalisation of the closure of the image of the map \eqref{eq:KM2Map} (which is a finitely-generated free abelian subgroup), so it suffices to determine the kernel of this map.

\vspace{1ex}

\noindent$\mathbf{C_2.}$ This is generated by the \emph{non}-hyperelliptic involution, having 2 fixed points. It therefore has
\begin{align*}
\kappa_{-1}^K &= -1 \emptyset_2(x_2) + 2 \frac{1 - \tfrac{1}{2} \emptyset_2(x_2)}{x_2}\\
 &= x_2-1\\
\kappa_0^K &= \frac{-2-2}{2} \emptyset_2(x_2) + 2 = 2x_2-2\\
\kappa_i^K &=  2 x_2^i \quad i \geq 1.
\end{align*}

\vspace{1ex}

\noindent$\mathbf{C_3.}$
This subgroup has fixed point data $(z_1, z_2, z_3, z_4 ; 3,3,3,3 ; 1,1,2,2)$, and so has
\begin{align*}
\kappa_{-1}^K &= -1 \emptyset_3(x_3) + 2\frac{1 - \tfrac{1}{3} \emptyset_3(x_3)}{x_3} + 2 \frac{\emptyset_2([2](x_3)) - \tfrac{2}{3} \emptyset_3(x_3)}{x_3}\\
&= -x_3^2 + 3x_3-1\\
\kappa_0^K &= \frac{-2-4}{3} \emptyset_3(x_3) + 4 = -2 x_3^2 + 6x_3 - 2\\
\kappa_i^K &= 2 x_3^i + 2[2](x_3)^i \quad i \geq 1.
\end{align*}
These are all invariant under the transformation $x \mapsto [2](x)$, so lie in the rank 2 submodule generated by the unit and $x+[2](x) = 3x_3-x_3^2$. (This is because the normaliser of $C_3 \leq \Gamma_2$ contains an element acting on this subgroup by squaring ($=$ inversion), but it is also easily observed to be the case from the formulas.)

\vspace{1ex}

\noindent$\mathbf{C_5.}$
This subgroup has fixed point data $(z_1, z_2, z_3 ; 5,5,5 ; 1,1,3)$, and so has
\begin{align*}
\kappa_{-1}^K &= -1 \emptyset_5(x_5) + 2\frac{1 - \tfrac{1}{5} \emptyset_5(x_5)}{x_5} +  \frac{\emptyset_3([2](x_5)) - \tfrac{3}{5} \emptyset_5(x_5)}{x_5}\\
 &= -x_5^4 + 5x_5^3 - 9x_5^2 + 7x_5 - 1\\
\kappa_0^K &= \frac{-2-3}{5} \emptyset_5(x_5) + 3 = -x_5^4 + 5x_5^3 - 10x_5^2 + 10x_5 - 2\\
\kappa_i^K &= 2 x_5^i + [2](x_5)^i \quad i \geq 1.
\end{align*}

\vspace{1ex}

\noindent$\mathbf{C_8.}$
This subgroup has fixed point data $(z_1, z_2, z_3 ; 2,8,8 ; 1,3,1)$, and so has
\begin{align*}
\kappa_{-1}^K &= -1 \emptyset_8(x_8) + \emptyset_4([2](x_8))\frac{1 - \tfrac{1}{2} \emptyset_2(x_8)}{x_8} +  \frac{\emptyset_3([3](x_8)) - \tfrac{3}{8} \emptyset_8(x_8)}{x_8} + \frac{1-\tfrac{1}{8} \emptyset_8(x_8)}{x_8}\\
&= x_8^7 - 7x_8^6 + 22x_8^5 - 40x_8^4 + 45x_8^3 - 31x_8^2 + 12x_8 - 1 \\
\kappa_0^K &= \frac{-2-6}{8} \emptyset_8(x_8) + \emptyset_4([2](x_8)) + 2\\
& = x_8^7 - 7x_8^6 + 22x_8^5 - 40x_8^4 + 46x_8^3 - 34x_8^2 + 16x_8 - 2\\
\kappa_i^K &= \emptyset_4([2](x_8)) \cdot x_8^i +  [3](x_8)^i + x_8^i \quad i \geq 1.
\end{align*}
These are all invariant under the transformation $x \mapsto [3](x)$, so lie in the rank 5 submodule generated by the unit and $x+[3](x), (x+[3](x))^2, (x+[3](x))^3, (x+[3](x))^4$.  (Again, this is explained by the fact that the normaliser of $C_8 \leq \Gamma_2$ contains an element acting on this group by cubing.)

\vspace{1ex}

We complete the calculation as follows. The target of \eqref{eq:KM2Map} is a free abelian group of rank $1+(2-1)+(3-1)+(5-1)+(8-1) = 15$, but the symmetry considerations for $C_3$ and $C_8$ show that the image of \eqref{eq:KM2Map} lies in a subgroup of rank $1+(2-1) + (3-2) + (5-1) + (8-4) = 11$. It is elementary to check (by computer) that the images under the map \eqref{eq:KM2Map} of the 11 elements
\begin{equation}\label{eq:RationallySpanningSet}
1, \kappa^K_{-1}+1, (\kappa^K_{-1}+1)^2, \ldots, (\kappa^K_{-1}+1)^9, \kappa^K_1
\end{equation}
are linearly independent. The ring structure is determined by the relations
\begin{align*}
0 =&(\kappa^K_{-1}+1)^{10} - 23(\kappa^K_{-1}+1)^9 + 238(\kappa^K_{-1}+1)^8 - 1457(\kappa^K_{-1}+1)^7\\
& + 5822(\kappa^K_{-1}+1)^6 - 15759(\kappa^K_{-1}+1)^5 + 28910(\kappa^K_{-1}+1)^4\\
& - 34670(\kappa^K_{-1}+1)^3 + 24660(\kappa^K_{-1}+1)^2 - 7920(\kappa^K_{-1}+1),\\
99(\kappa^K_1)^2 =& 594\kappa^K_1 -191118(\kappa^K_{-1}+1) + 525769(\kappa^K_{-1}+1)^2 - 647446(\kappa^K_{-1}+1)^3\\
& + 464979(\kappa^K_{-1}+1)^4 - 212833(\kappa^K_{-1}+1)^5 + 63526(\kappa^K_{-1}+1)^6\\
& - 12062(\kappa^K_{-1}+1)^7 + 1330(\kappa^K_{-1}+1)^8 - 65(\kappa^K_{-1}+1)^9,\\
99 \kappa_1^K (\kappa_{-1}^K+1) =& 198 \kappa_1^K  -7266(\kappa^K_{-1}+1) + 15109(\kappa^K_{-1}+1)^2 - 14926(\kappa^K_{-1}+1)^3\\
& + 9318(\kappa^K_{-1}+1)^4 - 4042(\kappa^K_{-1}+1)^5 + 1246(\kappa^K_{-1}+1)^6 \\
&- 263(\kappa^K_{-1}+1)^7 + 34(\kappa^K_{-1}+1)^8 - 2(\kappa^K_{-1}+1)^9.
\end{align*}
(The factors 99 indicate that the span of the images of the 11 elements \eqref{eq:RationallySpanningSet} is not quite a subring, but as we will shortly rationalise this will not matter). We therefore obtain
$$\bQ \otimes R\K^0(\mathcal{M}_2)/\sqrt{0} = \bQ \otimes \bZ[[\kappa_{-1}^K+1, \kappa_1^K]]/(\text{these three relations}).$$

A further consequence, because the symmetry considerations come from normalisers of $C_3$ and $C_8$ in $\Gamma_2$ so $\K^0(\mathcal{M}_2) \to R(C_2)^\wedge_I \times_\bZ R(C_3)^\wedge_I \times_\bZ R(C_5)^\wedge_I \times_\bZ R(C_8)^\wedge_I$ lands in the invariant part, is that rationally and modulo nilpotents all of $\K^0(\mathcal{M}_2)$ is tautological:
$$\bQ \otimes R\K^0(\mathcal{M}_2)/\sqrt{0} = \bQ \otimes \K^0(\mathcal{M}_2)/\sqrt{0}.$$

\begin{rem}
It would be interesting to compare this with the calculation by Edidin--Hu \cite{EdidinHu} of the algebraic $K$-theory of the algebraic stack $\mathcal{M}_2$.
\end{rem}

\section{$\MU_\bQ$-theory}\label{sec:MUQ}

We return to the identification $R\MU^*_\bQ(\mathcal{M}_g) \cong R\H^*(\mathcal{M}_g;\bQ) \otimes_\bZ \MU^*$ from \eqref{eq:MURatTaut}, induced by the Chern--Dold character. Our original interest in Mumford--Morita--Miller classes and tautological rings in generalised cohomology theories, and in particular complex cobordism, arose from the hope that the complicated descriptions of the tautological rings $R\H^*(\mathcal{M}_g;\bQ)$ might have a more transparent geometric meaning in complex cobordism. Ultimately, this does not seem to be the case. But here we record a detailed example which shows that at least some interesting features occur.

Consider $\mathcal{M}_6$. Faber \cite{Faber} has calculated $R\H^*(\mathcal{M}_6;\bQ)$ completely, and it is generated by $\kappa_1$ and $\kappa_2$ subject to the relations
\begin{align*}
127\kappa_1^3 + 2304 \kappa_1 \kappa_2 &=0\\
113 \kappa_1^4 - 36864 \kappa_2^2 &=0.
\end{align*}
The classes $\kappa_3$ and $\kappa_4$ are nonzero and are described by
\begin{align*}
\kappa_3 &= \tfrac{5}{2304} \kappa_1^3\\
\kappa_4 &= -\tfrac{5}{73728} \kappa_1^4;
\end{align*}
the higher $\kappa_i$ vanish. (These expressions differ by some signs from Faber's, as we have defined Mumford--Morita--Miller classes using the Chern class of the vertical tangent bundle, whereas in algebraic geometry the cotangent bundle is used.) 

Using \eqref{eq:ChOnKappas} we find
\begin{align*}
\ch(\kappa_1^{MU}) &= \kappa_1 - \tfrac{1}{2}[\bP^1]\kappa_2 + (\tfrac{1}{2} [\bP^1]^2 - \tfrac{1}{3}[\bP^2])\kappa_3 + (-\tfrac{5}{8}[\bP^1]^3 + \tfrac{5}{6}[\bP^1][\bP^2] - \tfrac{1}{4}[\bP^3])\kappa_4\\
\ch(\kappa_2^{MU}) &= \kappa_2 - [\bP^1] \kappa_3 + (\tfrac{5}{4}[\bP^1]^2 - \tfrac{2}{3}[\bP^2])\kappa_4,
\end{align*}
and using these and the above relations we see that $R\MU^*_\bQ(\mathcal{M}_6)$ is generated as a $\MU^*_\bQ$-algebra by $\kappa_1^{MU}$ and $\kappa_2^{MU}$ subject to the relations
\begin{align}
260352\KMU_1 \KMU_2+14351(\KMU_1)^3-72600[\P^1](\KMU_2)^2 &=0 \tag{$R1$}\\
2032(\KMU_2)^2+113\KMU_2(\KMU_1)^2&=0 \tag{$R2$}\\
(\KMU_2)^3 &=0 \tag{$R3$}
\end{align}
There are also relations
\begin{align*}
(\KMU_{1})^3+32\KMU_{1}\KMU_{2}+352\KMU_{3} &=0 \tag{$D_3$}\\
5 (\KMU_{2})^2+226 \KMU_{4} &=0 \tag{$D_4$}\\
\KMU_{i} &=0 \text{ for } i \geq 5 \tag{$D_i$}
\end{align*}
expressing the remaining $\KMU_i$ in terms of $\kappa_1^{MU}$ and $\kappa_2^{MU}$.%

\begin{rem}
We have
\begin{align*}
&-113\kappa_1^{MU} \kappa_2^{MU}(R1) + ( 14351(\kappa_1^{MU})^2 + 2288\kappa_2^{MU} )(R2) \\
&\quad\quad =  4649216(\kappa_2^{MU})^3 + 8203800[\bP^1]\kappa_1^{MU}(\kappa_2^{MU})^3
\end{align*}
so modulo ($R1$) and ($R2$) we have that $(\kappa_2^{MU})^3$ is a nonzero rational multiple of $[\bP^1]\kappa_1^{MU}(\kappa_2^{MU})^3$ so that when passing to ordinary cohomology, which has the effect of setting $[\P^1]=0$, ($R3$) becomes redundant. However, ($R3$) is not redundant as a relation in $\MU_\bQ^*[\KMU_{1}, \KMU_{2},\ldots]$. On the other hand, the above shows that
$$(\kappa_2^{MU})^3(4649216 + 8203800[\bP^1]\kappa_1^{MU}) \text{ lies in the ideal generated by ($R1$) and ($R2$)}$$
and $[\bP^1]\kappa_1^{MU} \in (\KMU_{1}, \KMU_{2},\ldots)$ so $4649216 + 8203800[\bP^1]\kappa_1^{MU}$ is a unit in the complete ring $\MU_\bQ^*[[\KMU_{1}, \KMU_{2},\ldots]]$, so the relation ($R3$) is redundant here.

We also have
$$127\kappa_1^{MU} (R2) - \kappa_2^{MU}(R1)  =  72600[\bP^1](\kappa_2^{MU})^3    -2288 \kappa_1^{MU}(\kappa_2^{MU})^2,$$
so ($R3$) may be replaced by 
\begin{align*}
\kappa_1^{MU}(\kappa_2^{MU})^2 &=0. \tag{R3$^\prime$}
\end{align*}
\end{rem}

\begin{rem}
An advantage of the $\MU_\bQ$ tautological ring is that it carries an action of the Landweber--Novikov operations, so the ideal of tautological relations must be preserved by these operations. The formulas in Lemma \ref{lem:LNOps} can be used to evaluate these. The operation $s_1$ is a derivation, and is given by $s_1(\kappa_n^{MU}) = n \cdot \kappa_{n+1}^{MU}$ for $n \geq 0$, and by $s_1([\bP^1])=-2$. For example, applying $s_1$ to the relation $(D_3)$  and then using $(D_3)$ again to eliminate $\KMU_{3}$, gives a relation expressing $\KMU_{4}$ in terms of $\kappa_1^{MU}$ and $\kappa_2^{MU}$. Further applications of $s_1$ express $\kappa_n^{MU}$ for $n \geq 3$ in terms of $\kappa_1^{MU}$ and $\kappa_2^{MU}$.
 
In fact, the ideal of tautological relations in the ring $\MU_\bQ^*[\KMU_{1}, \KMU_{2},\ldots]$ is generated by $(R1)$, $(D_3)$, and the operation $s_1$ iteratedly applied to these. We omit this verification.
\end{rem}

Finally, we may express the fundamental tautological class in terms of these generators as
\begin{align*}
\KMU_{-1} &= -5[\P^{1}] +\left(-\tfrac{1}{4}[\P^{1}]^2+\tfrac{1}{3}[\P^{2}]\right)\KMU_{1}\\
&\quad +\left(\tfrac{1}{8}[\P^{1}]^3-\tfrac{1}{3}[\P^{2}][\P^{1}]+\tfrac{1}{4}[\P^{3}]\right)\KMU_{2}\\
&\quad+\left(\tfrac{1}{176}[\P^{1}]^4+\tfrac{1}{99}[\P^{2}]^2+\tfrac{1}{44}[\P^{3}][\P^{1}]-\tfrac{1}{44}[\P^{2}][\P^{1}]^2-\tfrac{1}{55}[\P^{4}]\right)\KMU_{2}\KMU_{1}\\
&\quad+\left(\tfrac{1}{3168}[\P^{2}]^2-\tfrac{1}{1408}[\P^{2}][\P^{1}]^2+\tfrac{1}{5632}[\P^{1}]^4-\tfrac{1}{1760}[\P^{4}]+\tfrac{1}{1408}[\P^{3}][\P^{1}]\right)(\KMU_{1})^3\\
&\quad+\big(\tfrac{1}{226}[\P^{4}][\P^{1}]-\tfrac{5}{7232}[\P^{1}]^5+\tfrac{5}{1356}([\P^{2}][\P^{1}]^3+[\P^{2}][\P^{3}]-[\P^{5}]-[\P^{1}][\P^{2}]^2)\\
&\quad\quad-\tfrac{15}{3616}[\P^{3}][\P^{1}]^2\big)(\KMU_{2})^2.
\end{align*}
It is not clear to me what to make of such expressions.

\section{Morita's method does not give relations in $\MU$-theory}

The relations of Faber in $R\H^*(\mathcal{M}_6;\bQ)$ used in the previous section can also be obtained by a method introduced by Morita \cite{MoritaJacobianI}, which was further developed in \cite{R-WRelations} and then \cite{CGJZ}. This method can be directly implemented in $R\MU^*_\bQ(\mathcal{M}_g)$ instead, where we assume it reproduces the relations of the previous section.

Our goal here is to explain why this method does not seem to yield relations in $\MU^*(\mathcal{M}_g) \otimes \bQ$. To do so, we explain Morita's idea. Consider the torus bundle
$$p : J_{2g} =(\bR^{2g} / \bZ^{2g}) \times_{\Sp(2g, \bZ)} E\Sp(2g, \bZ) \lra B\Sp(2g, \bZ),$$
which is equipped with a canonical section $s : B\Sp(2g, \bZ) \to J_{2g}$. Morita shows that there is a unique line bundle $V \to J_{2g}$ characterised by the following properties of its Chern class $\Omega := c_1(V)$,
\begin{enumerate}[(i)]
\item $s^*(\Omega)$ is trivial,

\item $\Omega$ restricted to a fibre $\bR^{2g} / \bZ^{2g}$ is twice the class of the symplectic form.
\end{enumerate}
He then shows that the class $\Omega$ has the property that $\Omega^{g+1}$ is torsion (and explicitly describes an integer which annihilates it). Unfortunately the direct analogue of this cannot be used to obtain relations in $\MU^*(\mathcal{M}_g)\otimes \bQ$, by the following theorem.

\begin{thm}
No power of $\Omega^{MU} := c_1^{MU}(V) \in \MU^2(J_{2g})$ is torsion.
\end{thm}
\begin{proof}
There are bundle maps
\begin{equation*}
\begin{tikzcd}
E C_6 \times_{C_6} T^2 \rar \dar& E\SL(2,\bZ) \times_{\SL(2,\bZ)} T^2 \rar \dar& E\Sp(2g, \bZ) \times_{\Sp(2g, \bZ)} T^{2g} \dar\\
BC_6 \rar & B\SL(2, \bZ) \rar & B\Sp(2g, \bZ).
\end{tikzcd}
\end{equation*}
It is enough to show that $\Omega^{MU}$ pulled back to $E C_6 \times_{C_6} T^2$ has the property that no power is torsion. This pull back still satisfies (i) and (ii).

Let us write $\pi : E \to B$ for the bundle $E C_6 \times_{C_6} T^2 \to BC_6$, $s : B \to E$ for the canonical section, and $L \to B$ for the canonical complex line bundle (where $C_6$ acts as the sixth roots of unity). As we explained in the proof of Theorem \ref{thm:elliptic}, there is an isomorphism $s^*(T_\pi E) \cong L$, and so we may find an embedding $e : L \to E$ as a tubular neighbourhood of the zero section. There is then a corresponding collapse map
$$c : E \lra L^+$$
to the fibrewise 1-point compactification of $L$ over $B$. There is furthermore a cofibre sequence
\begin{equation}\label{eq:CofSeq}
B \overset{s_\infty}\lra L^+ \overset{q}\lra \mathrm{Th}(L \to B)
\end{equation}
where $s_\infty$ is the section at infinity. We define a class
$$\nu := c^*(q^*(u)) \in \H^2(E;\bZ)$$
where $u \in \H^2(\mathrm{Th}(L);\bZ)$ is the Thom class. Restricted to a fibre $T$ of $E$ the class $\nu$ is a generator (as $c$ had degree 1 on each fibre). 

Recall that $C_6 \subset \SL(2, \bZ)$ is generated by the element $T=\bigl(\begin{smallmatrix}
0&-1\\ 1&1
\end{smallmatrix} \bigr)$, and the matrix $T - I = \bigl(\begin{smallmatrix}
-1&-1\\ 1&0
\end{smallmatrix} \bigr)$ is invertible. It then follows from the standard periodic resolution of $\bZ$ over $\bZ[C_6]$ that
$$\H^*(C_6 ; \H^1(T;\bZ))=0,$$
and hence the Serre spectral sequence for $\pi : E \to B$ collapses. In addition we see that $c : E \to L^+$ induces an isomorphism on $E^2$-pages of the associated Serre spectral sequences, so $c$ induces an isomorphism on cohomology. Thus it is a stable homotopy equivalence, so also induces an isomorphism in $\MU$-theory. The cofibre sequence \eqref{eq:CofSeq} is split by the projection $L^+ \to B$, so we have stable homotopy equivalences
$$\Sigma^\infty_+ E \simeq \Sigma^\infty_+ L^+ \simeq \Sigma^\infty_+ B \vee \Sigma^\infty \mathrm{Th}(L)$$
and so $\MU^i(E) \cong \MU^i(BC_6) \oplus \MU^{i-2}(BC_6)$, which is torsion-free (Theorem \ref{thm:MUCnFlat}). Thus it will be enough to show that the pullback of $\Omega^{MU}$ to $\MU^*(E)$ is not nilpotent, and for this it is enough to show that the pullback of $\Omega$ to $\H^*(E;\bZ)$ is not nilpotent.

By the collapse of the Serre spectral sequence for $\pi$, we have that $\H^2(E;\bZ) = \bZ\{\nu\} \oplus \bZ/6\{\pi^*(x)\}$ where $x = c_1(L) \in \H^2(C_6;\bZ)$ and $\nu$ we constructed above. The identity $s^*(\nu) = x$ follows from the definition of $\nu$, and hence $\omega := 2(\nu - \pi^*(x))$ is the unique class satisfying
\begin{enumerate}[(i)]
\item $s^*(\omega)=0$,

\item $\omega \vert_{T^2}$ is twice a generator of $\H^2(T^2;\bZ)$,
\end{enumerate}
so is the pullback of $\Omega$ to $E$. The identity $\nu^2 = \pi^*(x) \cdot \nu$ holds (by definition of the Euler class in terms of the Thom class), and so
$$\omega^d = -(-2)^d\nu\cdot\pi^*(x)^{d-1} +(-2\pi^*(x))^{d} \in \H^{2d}(E;\bZ) = \bZ/6\{\pi^*(x)^d, \nu \cdot \pi^*(x)^{d-1}\}$$
which is not zero, so $\omega$ is not nilpotent.
\end{proof}

\bibliographystyle{plain}
\bibliography{MainBib}  

\end{document}